\documentclass[11pt,a4paper]{amsart}

\title[A Ronkin type function for coamoebas]{A Ronkin type function for coamoebas}

\author{Petter Johansson \& H{\aa}kan Samuelsson Kalm}
\thanks{The second author was partially supported by the Swedish Research Council.}
\address{Petter Johansson, Department of Mathematics, Stockholm University, SE-106 91 Stockholm, Sweden}
\email{petterj@math.su.se}
\address{H{\aa}kan Samuelsson Kalm, Department of Mathematical Sciences, Division of Mathematics, University of Gothenburg and 
Chalmers University of Technology, SE-412 96 G\"{o}teborg, Sweden}
\email{hasam@chalmers.se}
\date{\today}

\usepackage[english]{babel}
\usepackage{amsmath}
\usepackage{amsthm,amssymb,latexsym}
\usepackage{url,ae}
\usepackage{t1enc}
\usepackage{mathrsfs}
\usepackage{graphicx}
\usepackage{eufrak} 
\usepackage{geometry}
\usepackage{esint}
\newtheorem{proposition}{Proposition}[section]
\newtheorem{theorem}[proposition]{Theorem}
\newtheorem{lemma}[proposition]{Lemma}
\newtheorem{porism}[proposition]{Porism}
\newtheorem{corollary}[proposition]{Corollary}

\theoremstyle{definition}

\newtheorem{example}[proposition]{Example}
\newtheorem{remark}[proposition]{Remark}

\numberwithin{equation}{section}

\newcommand{\C}{\mathbb{C}}
\newcommand{\debar}{\bar{\partial}}

\newcommand{\R}{\mathbb{R}}

\newcommand{\hol}{\mathscr{O}}

\newcommand{\real}{\mathfrak{R}\mathfrak{e}}
\newcommand{\imag}{\mathfrak{I}\mathfrak{m}}

\newcommand{\B}{\mathbb{B}}

\def\newop#1{\expandafter\def\csname #1\endcsname{\mathop{\rm #1}\nolimits}}
\newop{span}

\begin{document}
\nocite{*}
\bibliographystyle{plain}

\begin{abstract}
The Ronkin function plays a fundamental role in the theory of amoebas.
We introduce an analogue of the Ronkin function in the setting of coamoebas.
It turns out to be closely related to a certain
toric arrangement known as the shell of the coamoeba and we use our
Ronkin type function to obtain some properties of it.
\end{abstract}

\maketitle
\thispagestyle{empty}

\section{Introduction}
Let $f$ be an exponential polynomial in $\C^n=\R^n+i\R^n$ of the form
\begin{equation*}
f(z)=\sum_{\alpha\in A} c_{\alpha} e^{\alpha\cdot z},
\end{equation*}
where $A\subset \mathbb{Z}^n$ is a finite subset, $c_{\alpha}\in \C^*:=\C\setminus\{0\}$, and $\alpha\cdot z=\sum_j\alpha_jz_j$.
We denote by $\mathcal{A}_f$ and $\mathcal{A}_f'$ the natural projections of $f^{-1}(0)$ on $\R^n$ and $i\R^n$ respectively.
The sets $\mathcal{A}_f$ and $\mathcal{A}_f'$ are known respectively as the \emph{amoeba} and the \emph{coamoeba} 
of $f$.\footnote{In perhaps more standard terminology, $\mathcal{A}_f$ and $\mathcal{A}_f'$ are the amoeba and coamoeba respectively 
of the \emph{Laurent polynomial} $\tilde{f}(\zeta):=\sum_{\alpha\in A} c_{\alpha}\zeta^{\alpha}$
on the complex torus $(\C^*)^n$.} 
Since $f$ is periodic in each imaginary direction with periodicity $2\pi$ we can also think of $f$ as a function
on $\C^n/(2\pi i\mathbb{Z})^n$ and we can view $\mathcal{A}_f'$ as a subset of the real 
torus $\mathbb{T}^n:=i\R^n/(2\pi i\mathbb{Z})^n$.

The Ronkin function \cite{ronkin} 
associated to $f$ is the function $R_f\colon \R^n\to\R$ defined by 
\begin{equation*}
R_f(x):= \frac{1}{(2\pi)^n} \int_{y \in [0,2\pi]^n} \log|f(x+iy)| \,dy,
\end{equation*}
and it plays a crucial role in the study of $\mathcal{A}_f$.
Since $\log|f(z)|$ is plurisubharmonic on $\C^n$, the Ronkin function is convex on $\R^n$,
and since $\log|f(z)|$ is pluriharmonic as long as $\real\, z\notin \mathcal{A}_f$, 
it is straightforward to check that $R_f$ is affine on $\R^n\setminus \mathcal{A}_f$. The gradient 
$\nabla R_f\colon\R^n\to\R^n$ is thus locally constant outside $\mathcal{A}_f$. In fact, see \cite{FoPaTs},
$\nabla R_f$ induces an injective map from the set of connected components of $\R^n\setminus \mathcal{A}_f$ to
$\Delta_f\cap \mathbb{Z}^n$, where $\Delta_f$ is the \emph{Newton polytope} of $f$, i.e., the convex hull of $A$ in $\R^n$. 
In particular, the number of complement components of
$\mathcal{A}_f$ is at most $|\Delta_f\cap \mathbb{Z}^n|$, and this bound turns out to be sharp, \cite{Rull}.

Let $\mathcal{E}_f$ be the set of complement components of $\mathcal{A}_f$. If $E\in \mathcal{E}_f$,
then, in view of the preceding
paragraph, it is natural to label $E$ by $\alpha$,
$E=E_{\alpha}$, where $\alpha=\nabla R_f(x)$ for any $x\in E$. 
Let $a_{\alpha}:=R_f(x)-\alpha\cdot x$ for any $x\in E_{\alpha}$ and let
\begin{equation*}
\sigma(x):=\max\{a_{\alpha} + \alpha\cdot x;\, E_{\alpha}\in \mathcal{E}_f\}.
\end{equation*}
Then $\sigma$ is a tropical polynomial and the set where $\sigma$ is non-smooth is called 
the \emph{spine} of $\mathcal{A}_f$. It is proved in \cite{PaRull} that the spine is contained in 
$\mathcal{A}_f$ and that it is a deformation retract of $\mathcal{A}_f$.

Let us now look at the Ronkin function in a slightly different way. Let $T_f:=dd^c\log|f(z)|$ be the Lelong current, 
where $d^c:=(\partial-\debar)/2\pi i$ so that $dd^c=i\partial\debar/\pi$.
One can view the Ronkin function as a $dd^c$-potential to an average of translates of $T_f$. More precisely,
consider $R_f$ as a function on $\C^n$ that is independent of $y=\imag\,z$ and let, for $w\in \C^n$, 
$T_f^{w}(z):=T_f(z+w)$. Then one can check that,
in the sense of currents in $\C^n$,
\begin{equation*}
dd^c R_f=\frac{1}{(2\pi)^n} \int_{s\in [0,2\pi]^n} T_f^{is}\, ds.
\end{equation*}
This observation is the motivation for our analogue of the Ronkin function for coamoebas.

The first step is to take an appropriate average of currents $T_f^r$ for $r\in \R^n$. 
Let $\B^k(a,R)$ be the ball in $\R^k$ with radius $R$ centered at $a$ and let 
$\textrm{Vol}_k$ denote the $k$-dimensional volume of a set in $\R^n$; set
also $\B^k:=\B^k(0,1)$. In Section~\ref{sec:T}
we prove that for any exponential polynomial $f$ 
the following limit exists and defines a positive $(1,1)$-current $\check{T}_f$ in $\C^n$ 
whose coefficients are independent of $x=\real\, z$.
\begin{equation}\label{rar}
\lim_{R\to\infty} \frac{1}{\kappa^{n-1}(R)} \int_{r\in \B^n(0,R)} T_f^r \, dr
\end{equation}
where we, for future reference and convenience, have introduced the notation
$\kappa^{n-1}(R):=R^{n-1}\textrm{Vol}_{n-1}(\B^{n-1})=\textrm{Vol}_{n-1}(\B^{n-1}(0,R))$.
Since $T_f$ is $d$-closed and periodic in each imaginary direction the current $\check{T}_f$ has the same properties.

Actually, we obtain an explicit formula for $\check{T}_f$ in terms of data from the Newton polytope $\Delta_f$.
To describe this formula assume first that $\Delta_f=:\Gamma$ is one-dimensional. Then, since $A\subset \mathbb{Z}^n$,
one can choose $\alpha_0, \beta\in \mathbb{Z}^n$ and integers $0=k_0<k_1<\cdots < k_{\ell}$ such that
\begin{equation}\label{rut}
A=\Gamma\cap A =\{\alpha_0, \alpha_0 + k_1\beta,\cdots, \alpha_0 + k_{\ell}\beta\}.
\end{equation}
We can thus write
\begin{equation*}
f(z)=e^{\alpha_0\cdot z} \sum_{j=0}^{\ell} c_j \big(e^{\beta\cdot z} \big)^{k_j}=e^{\alpha_0\cdot z}P(e^{\beta\cdot z}),
\end{equation*}
where $c_j:=c_{\alpha_0+k_j\beta}$ and $P(w):=\sum_{j=0}^{\ell} c_j w^{k_j}$ is a one-variable polynomial. 
Let $\{a_j\}$ be the distinct zeros of $P$ and let $\{d_j\}$ be the associated multiplicities. Then,
by Proposition~\ref{edgeprop}, we have
\begin{equation}\label{eq:1dimformel}
\check{T}_f(x+iy)=\sum_j d_j \frac{\beta\cdot dx}{|\beta|} \wedge [\beta\cdot y=\arg a_j],
\end{equation}
where we consider $\arg$ as a multivalued function and 
$[\beta\cdot y=\arg a_j]$ is the current of integration over the real hypersurfaces $\{\beta\cdot y=\arg a_j\}$
that are oriented so that $(\beta\cdot dx)\wedge [\beta\cdot y=\arg a_j]$ becomes a positive current in $\C^n$. 

In the general case when the dimension of $\Delta_f$ is arbitrary, let $\mathcal{F}_1(\Delta_f)$ be the 
edges, i.e., the one-dimensional faces of $\Delta_f$, and let, for $\Gamma\in\mathcal{F}_1(\Delta_f)$,
\begin{equation*}
f_{\Gamma}(z):=\sum_{\alpha\in \Gamma\cap A} c_{\alpha}e^{\alpha\cdot z}.
\end{equation*}
The Newton polytope of $f_{\Gamma}$ is one-dimensional and $\check{T}_{f_{\Gamma}}$ is given 
by a formula like \eqref{eq:1dimformel}. An explicit formula for $\check{T}_f$ thus follows from the next
result. Here 
$\gamma_{\Delta_f}(\Gamma)$ is the \emph{external angle} of $\Gamma$ in $\Delta_f$; see the next section for 
a definition.

\begin{theorem}\label{Thuttsats} 
Let $f$ be an exponential polynomial with associated 
Newton polytope $\Delta_f$. Then
\begin{equation*}
\check{T}_f = \sum_{\Gamma \in\mathcal{F}_1(\Delta_f)}
\gamma_{\Delta_f}(\Gamma) \check{T}_{f_{\Gamma}}.
\end{equation*}
\end{theorem}

Theorem~\ref{Thuttsats} is a weak version of Theorem~\ref{Thuttsats}', which also
shows that the limit \eqref{rar} exists.
We remark that if $f$ is not an exponential polynomial but has similar growth and 
still is periodic in each imaginary direction, then we do not know 
whether the limit \eqref{rar} exists or not.

It follows from Theorem~\ref{Thuttsats} and \eqref{eq:1dimformel} that the support of 
$\check{T}_f$ is of the form $\R^n+i\mathcal{H}_f$, where $\mathcal{H}_f$
is a union of real hyperplanes of the form $\{\beta\cdot y=\arg a\}$ with $\beta\in \mathbb{Z}^n$ and $a\in\C$.
A hyperplane arrangement of this type, that repeats itself periodically, is
called a \emph{toric arrangment}, see, e.g., \cite{toric}.
Notice that $\mathcal{H}_f$ is 
the union of the coamoebas of $f_{\Gamma}$ for $\Gamma\in \mathcal{F}_1(\Delta_f)$.
This union is known as the \emph{shell} of the coamoeba of $f$, see \cite{petter}; 
the shell is the part of the \emph{phase limit set}, \cite{phaselimit}, that corresponds
to the edges of $\Delta_f$. Notice also that the hyperplanes 
constituting the shell come with a natural multiplicity, cf.\ \eqref{eq:1dimformel}.
The union of the shell and the coamoeba equals the closure of the
coamoeba, see, e.g., \cite{petter}. It is furthermore known,
\cite{Jens}, that if $C_1,\ldots,C_k$ are the distinct full-dimensional
complement components of the coamoeba, then there are $k$ distinct
complement components $E_1,\ldots,E_k$ of the shell such that $C_j\subset \overline{E}_j$ for every $j$.

\begin{figure}[ht]
\begin{center}
\includegraphics[height=4cm]{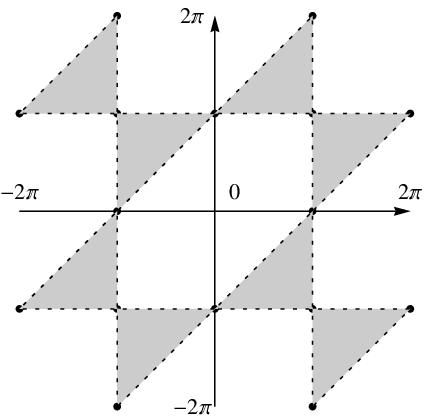}\qquad
\includegraphics[height=4cm]{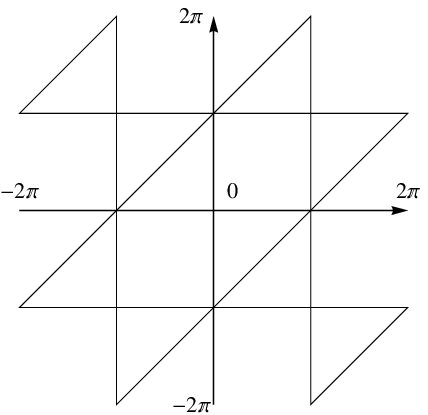}
\caption{\small{The shaded region on the left hand side is the coamoeba of $f(z)=1+e^{z_1}+e^{z_2}$, and the
right hand side shows its shell.}}
\label{figur1}
\end{center}
\end{figure}

The current $\check{T}_f$ is a positive and $d$-closed current in $\C^n$. It is well-known that then
there is a plurisubharmonic function $\varphi$ such that $dd^c\varphi=\check{T}_f$. Moreover, since 
the coefficients of $\check{T}_f$ are independent of $x=\real\, z$ there is such a $\varphi$ that only
depends on $y=\imag\, z$ and this function then must be convex. A convex function can of course not 
be periodic (unless it is constant), but since $\check{T}_f$ is periodic, 
we will see that there is a symmetric matrix $S_f$, explicitly
given in terms of $\Delta_f$ (see \eqref{Smatris}), such that the following holds.

\begin{theorem}\label{ronkinfunk}
There is convex and piecewise affine function $\varphi_f$ on $\C^n$ that only depends on $y=\imag\, z$
such that
\begin{itemize}
\item[(i)] $dd^c \varphi_f = \check{T}_f$,
\item[(ii)] $y\mapsto \varphi_f(y) - y^tS_fy$ is periodic, where we consider $y$ as a column vector.
\end{itemize}
\end{theorem}

This result is essentially Theorem~\ref{ronkin} below but there we define $\varphi_f$ explicitly; this 
explicit function is our analogue of the Ronkin function.

\smallskip

In analogy with the amoeba case, let us consider the gradient $\nabla\varphi_f$ of $\varphi_f$, where we now consider 
$\varphi_f$ as a function on $\R^n_y$. It is not hard to see that $\nabla\varphi_f$ is locally constant
outside of $\mathcal{H}_f$. 
Clearly, $\nabla\varphi_f$ is not periodic but it turns out that $\nabla\varphi_f$ induces a map
$\R^n/(2\pi\mathbb{Z})^n\to \R^n/L_{S_f}$ where $L_{S_f}$ is the ``skew'' lattice generated by
the columns of the matrix $4\pi S_f$. Moreover, this map is injective from the set of connected components
of $(\R^n\setminus \mathcal{H}_f)/(2\pi \mathbb{Z})^n$. 

\begin{figure}[ht]
\begin{center}
\includegraphics[height=4cm]{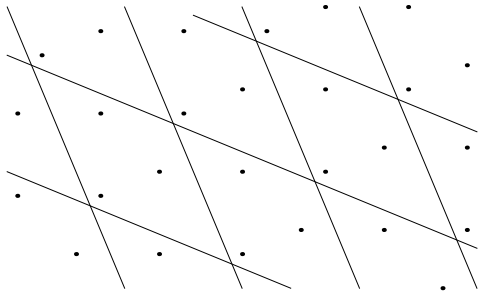}
\caption{\small{The intersection points of the lines make up the lattice $L_{S_f}$ of $f(z)=1+e^{z_1}+e^{z_2}$ 
and the dots are the images of the 
complement components of the shell of $\mathcal{A}'_f$ under $\nabla\varphi_f$;
cf.\ Figure~\ref{figur1} (right).}}
\label{figur2}
\end{center}
\end{figure}

As another application of our Ronkin type function we estimate the number of vertices of the 
toric arrangement $\mathcal{H}_f$ from below and from above.  
Let $\Gamma_1,\ldots,\Gamma_N$ be an enumeration of the edges of $\Delta_f$,
let $\beta_i\in \mathbb{Z}^n$ be parallel to $\Gamma_i$ so that the analogue of \eqref{rut} holds, 
and let, for $I=\{i_1,\ldots,i_n\}\subset \{1,\ldots,N\}$, $B_I$ be the matrix whose columns are
the $\beta_{i_k}$'s. Recall that the hyperplanes that constitute $\mathcal{H}_f$ come with a natural multiplicity; 
intersections of such
hyperplanes therefore also get a natural multiplicity (we consider all intersections as positive).
Let $\mathcal{I}_f$ be the set of points, repeated according to multiplicity, in $[0,2\pi)^n$ 
(or any other fundamental domain for $\mathbb{T}^n$) where $n$ hyperplanes from $\mathcal{H}_f$
meet transversely.

\begin{theorem}\label{estimate}
Let $M$ be the maximum of $\left|\det B_I\right|\prod_{k=1}^n\gamma_{\Delta_f}(\Gamma_{i_k})/|\beta_{i_k}|$
over all $I\subset \{1,\ldots,N\}$ 
and let $m$ be the minimum of the same quantity over all $I\subset \{1,\ldots,N\}$ such that $\det B_I\neq 0$.
Then
\begin{equation*}
\frac{2^n\det S_f}{M} \leq |\mathcal{I}_f| \leq \frac{2^n\det S_f}{m}.
\end{equation*}
\end{theorem}

These estimates turn out to be sharp in a simple example, see the end of Section~\ref{sec:ronkin}, 
``Example~\ref{ex}, continued''. In the case that $f$ is generic we also show that 
$|\mathcal{I}_f|$ as well as the number of connected components of
$(\R^n\setminus\mathcal{H}_f)/(2\pi\mathbb{Z})^n$ equals
\begin{equation}\label{last}
\sum_{\substack{I\subset \{1,\ldots,N\} \\ |I|=n}} \left|\det B_I\right|,
\end{equation}
see Corollary~\ref{genericCor} below. Notice that, in view of the paragraph preceding Figure~1, \eqref{last} is an
upper bound for the number of complement components of the closure of $\mathcal{A}'_f$ in the generic case.
However, this bound is quite rough.

\medskip



The paper is organized as follows. In the next section we recall some foundational facts and notions
about convex polytopes, mixed volumes, and quermassintegrals of polytopes and balls. In Section~\ref{sec:lelong}
we prove some estimates of Lelong type for the volume of the zero set of an exponential polynomial.
Section~\ref{sec:T} is the technical core of the paper; we prove that the current $\check{T}_f$ exists and we compute it 
explicitly in terms of data from the Newton polytope. In Section~\ref{sec:ronkin} we define, and give
an explicit formula for our Ronkin type function and we derive some consequences. In the final section 
we give a cohomological description and interpretation of the $S_f$-matrix.

\bigskip

{\bf Acknowledgments:} We would like to thank Bo Berndtsson for valuable discussions
on the topic of this paper.

\section{Normal cones, external angles, and quermassintegrals}


Let $\Delta$ be a convex polytope in $\R^n$. We denote by $\mathcal{F}_k(\Delta)$ the set of $k$-dimensional faces of
$\Delta$ and we set $\mathcal{F}(\Delta):=\bigcup_k \mathcal{F}_k(\Delta)$.
For $\Sigma\in \mathcal{F}_k(\Delta)$, $k=0,\ldots,\textrm{dim}\,\Delta-1$ 
we define the \emph{normal cone} $N_{\Sigma}$ of $\Sigma$ by
\begin{equation*}
N_{\Sigma}:=\{x\in\R^n; \, x\cdot (y-y')\geq 0, \, \forall y\in\Sigma, \, \forall y'\in \Delta\}.
\end{equation*}
We also let $N_{\Delta}$ be the linear subspace of $\R^n$ that is orthogonal to the affine subspace spanned by $\Delta$.
Notice that $x\in N_{\Sigma}$ if and only if $x$ is an outward normal vector to $\Delta$ at $\Sigma$,
and moreover, if $\Sigma \in \mathcal{F}_k(\Delta)$, then $N_{\Sigma}$ is an $n-k$-dimensional cone.

Assume that $\Delta$ is the convex hull of some finite set $A\subset \mathbb{Z}^n$. 
For $\Sigma\in\mathcal{F}_k(\Delta)$, where $k=0,\ldots,\textrm{dim}\,\Delta-1$, and $\epsilon>0$ 
we let
\begin{equation}\label{Nepsilon}
N_{\Sigma}^{\epsilon}:=\{x\in N_{\Sigma}; \, x\cdot (\alpha-\alpha')\geq \epsilon |x| |\alpha-\alpha'|, 
\, \forall \alpha\in\Sigma, \, \forall \alpha'\in  A \setminus \Sigma\};
\end{equation}
if $\textrm{dim}\,N_{\Sigma}=\textrm{dim}\,\Delta$ 
we set $N_{\Sigma}^{\epsilon}:=N_{\Sigma}$.
One can check that, for $\epsilon$ sufficiently small, $N_{\Sigma}^{\epsilon}$
is a subcone of $N_{\Sigma}$ of the same dimension and 
$\textrm{Vol}(N_{\Sigma}^{\epsilon}\cap\B^{\textrm{dim}\,N_{\Sigma}})\to \textrm{Vol}(N_{\Sigma}\cap\B^{\textrm{dim}\,N_{\Sigma}})$ 
as $\epsilon\to 0$. 
Notice that if $\Sigma$ and $\Sigma'$ are different faces then $N_{\Sigma}^{\epsilon}\cap N_{\Sigma'}^{\epsilon}=\{0\}$.
In fact, without loss of generality we may then assume that
there are $\alpha\in\Sigma\setminus \Sigma'$ and $\alpha'\in\Sigma'$. If
$x\in  N_{\Sigma}^{\epsilon}\cap N_{\Sigma'}^{\epsilon}$, then $x\cdot (\alpha-\alpha') \geq \epsilon |x| |\alpha-\alpha'|$
and $x\cdot (\alpha'-\alpha)\geq 0$. It immediately follows that $x=0$.

For each $\Sigma\in \mathcal{F}_k(\Delta)$ we let $\gamma_{\Delta}(\Sigma)$ be the \emph{external angle} of $\Sigma$
relative to $\Delta$,
see, e.g., \cite{BH}. We will use that it can be defined as
\begin{equation}\label{extang}
\gamma_{\Delta}(\Sigma):= \frac{\textrm{Vol}_{n-k}(N_{\Sigma}\cap \B^{n-k})}{\textrm{Vol}_{n-k}(\B^{n-k})}.
\end{equation}

\bigskip

Let $K_1,\ldots,K_n$ be convex compact sets in $\R^n$.
Then, by Minkowski's theorem, $t\mapsto \textrm{Vol}_n(t_1K_1+\cdots+t_nK_n)$ 
is an $n$-homogeneous
polynomial. The \emph{mixed volume} $V(K_1,\ldots,K_n)$ of $K_1,\ldots,K_n$ is defined as the coefficient of the monomial
$t_1\cdots t_n$ divided by $n!$. If $K$ and $Q$ are two convex compacta we let for $k=0,\ldots,n$
\begin{eqnarray*}
V_k(K,Q):=V(K_1,\ldots,K_n), \quad & & K_i=K, \,\, i=1,\ldots n-k, \\
& & K_j=Q, \,\, j=n-k+1,\ldots,n,
\end{eqnarray*}
which are known as the \emph{quermassintegrals}. By Steiner's formula we have that
\begin{equation}\label{steiner}
\textrm{Vol}_n(K+tQ)=\sum_{k=0}^n {n \choose k} V_k(K,Q) t^k.
\end{equation}

Now, $\Delta+t\B^n$ is the $t$-neighborhood
of $\Delta$ in $\R^n$ and $\textrm{Vol}_n(\Delta+t\B^n)$ can be computed explicitly in terms of the external angles
of the faces of $\Delta$. 
In fact, from \cite{mcmullen}, see also \cite{BH}, it follows that
\begin{equation*}
\textrm{Vol}_n(\Delta+t\B^n)=\sum_{k=0}^n \textrm{Vol}_{k}(\B^{k})\sum_{\Sigma\in \mathcal{F}_{n-k}(\Delta)}
\gamma_{\Delta}(\Sigma) \textrm{Vol}_{n-k}(\Sigma)\, t^{k}.
\end{equation*}
Hence, from \eqref{steiner} we get in particular that
\begin{equation}\label{kul}
V_{n-1}(\Delta,\B^n)=\frac{\textrm{Vol}_{n-1}(\B^{n-1})}{n}\sum_{\Sigma\in \mathcal{F}_{1}(\Delta)}
\gamma_{\Delta}(\Sigma) \textrm{Vol}_{1}(\Sigma),
\end{equation}
which we will have use for later.

\smallskip

We also recall how mixed volumes can be computed as integrals of mixed Monge-Amp\`{e}re expressions.
Let $u_1,\ldots,u_n$ be smooth convex functions defined on $\R^n$; we can consider the $u_j$'s
as convex functions on $\C^n$ that only depend on $x=\real\, z$. Then it is straightforward to verify that
\begin{equation}\label{prat}
dd^cu_1\wedge \cdots \wedge dd^cu_n=\frac{n!}{(2\pi)^n}
M(u_1,\ldots,u_n)dx_1\wedge dy_1\wedge \cdots\wedge dx_n\wedge dy_n,
\end{equation}
where $M$ is the symmetric multilinear operator such that 
$M(u,\ldots,u)$ is the determinant of the real Hessian of $u$.
It is well-known that the right hand side of \eqref{prat} can be extended, as a continuous operator,
to the space of $n$-tuples of continuous functions with the topology of locally uniform convergence
to the space of measures with the weak topology.

For a convex function $u$ on $\R^n$ we let $K_u$ be the set of $\eta\in \R^n$ such that 
$x\mapsto \eta\cdot x - u(x)$ is bounded from above.\footnote{In other words, $\eta$ is in $K_u$
if and only if the Legendre transform of $u$ at $\eta$ is finite.}
If $u_1,\ldots,u_n$ are convex functions on $\R^n$ such that $K_{u_1},\ldots,K_{u_n}$ are bounded,
then, by, e.g., \cite[Proposition~3 (iv)]{PaRull},
\begin{equation*}
V(K_{u_1},\ldots,K_{u_n})=\int_{x\in\R^n} M(u_1,\ldots,u_n)dx.
\end{equation*}
In view of \eqref{prat} we thus get
\begin{equation}\label{prat2}
V(K_{u_1},\ldots,K_{u_n}) = \frac{1}{n!}\int_{\R^n+i[0,2\pi]^n} dd^cu_1\wedge \cdots\wedge dd^cu_n.
\end{equation}

Now, let $f$ be an exponential polynomial with corresponding Ronkin function $R_f$ and Newton polytope $\Delta_f$.
By the proof of \cite[Theorem~4]{PaRull} we see that $K_{R_f}=\Delta_f$. Moreover, if $u(x)=|x|$, then $K_u=\overline{\B}^n$
as one can easily check. From \eqref{prat2} we thus get that
\begin{equation}\label{kul2}
V_{n-1}(\Delta_f,\B^n)=\frac{1}{n!} \int_{\R^n+i[0,2\pi]^n} dd^c R_f(x)\wedge (dd^c |x|)^{n-1}.
\end{equation}

\section{Lelong type estimates}\label{sec:lelong}
Let $f$ be a holomorphic function in $\C^n$, let $T_f=dd^c\log|f|$ be the associated Lelong current, 
and let $\xi$ be a $(n-1,n-1)$-test form in $\C^n$. Then 
\begin{equation*}
\R^n\ni r\mapsto \int_z T_f(z)\wedge \xi(z-r)
\end{equation*}
is a smooth function on $\R^n$. In this section we collect some results that will give us a certain control of
the growth of this type of functions when $f$ is an exponential polynomial. 

\begin{theorem}\label{fund}
Let $f$ be an exponential polynomial in $\C^n$ and let $T:=dd^c\log |f|$ be the associated
Lelong current. 
\begin{itemize}
\item[(i)] The function
\begin{equation*}
R\mapsto \frac{1}{R^{n-1}}\int_{\substack{|x|<R \\ y\in [0,2\pi]^n}} T\wedge (dd^c|x|^2)^{n-1}
\end{equation*}
is bounded and increasing.
\item[(ii)] For $r\in \R^n$, let $T^r$ be the translate $T^r(z)=T(z+r)$
and let $R'>0$ be fixed. Then
\begin{equation*}
r\mapsto \int_{\substack{|x|<R' \\ y\in [0,2\pi]^n}} T^r\wedge (dd^c|x|^2)^{n-1}
\end{equation*}
is a bounded function.
\end{itemize}
\end{theorem}

\begin{remark}
It follows from Theorem~\ref{fund} (ii) that if $D\subset \R^n$ is a domain such that $\textrm{Vol}_n(D\cap \{|x|<R\})\lesssim R^k$
and we let $\mathbf{D}(R)=\{x+iy\in \C^n; \, |x|<R,\, x\in D,\, y\in [0,2\pi]^n\}$, then
\begin{equation}\label{penna2}
R\mapsto \frac{1}{R^k}\int_{\mathbf{D}(R)} T\wedge (dd^{c}|x|^2)^{n-1}
\end{equation}
is bounded. In particular, for $D=\R^n$ it thus follows that \eqref{penna2} with $k=n$ is bounded. 
However, Theorem~\ref{fund} (i) says that we can do ``better'' in this case.
\end{remark}

\begin{lemma}\label{bob-lelong}
Let $\phi$ be a convex function on $\R^n$ such that $\phi(x)\leq A|x|+B$ for some constants $A$ and $B$.
Define $\varphi\colon\C^n\to \R$ by $\varphi(z)=\phi(\real\, z)$ and let $K\Subset \R^n$.
Then
\begin{equation*}
\int_{|x|<R,\, y\in K} dd^{c}\varphi\wedge (dd^{c}|x|)^{n-1}\leq C\cdot A,
\end{equation*}
where $C$ is a constant that only depends on $n$ and $K$.
\end{lemma}

Notice that  
$R \mapsto \int_{|x|<R,\, y\in K} dd^{c}\varphi\wedge (dd^{c}|x|)^{n-1}$ is increasing
since $dd^{c}\varphi\wedge (dd^{c}|x|)^{n-1}$ is positive. It thus follows from the lemma that
\begin{equation*}
\lim_{R\to\infty} \int_{|x|<R,\, y\in K} dd^{c}\varphi\wedge (dd^{c}|x|)^{n-1} =
\int_{|\R^n + iK} dd^{c}\varphi\wedge (dd^{c}|x|)^{n-1} <\infty.
\end{equation*}

\begin{proof}
Since $\phi$ is convex there is an affine function $\alpha\cdot x + \beta$ such that
$\phi(x)\geq \alpha\cdot x+\beta$. Let $\tilde{\varphi}(x)=\phi(x)-\alpha\cdot x-\beta$ and consider $\tilde{\varphi}$
as a function on $\C^n$ that is independent of $\mathfrak{Im}\, z$.
Then $0\leq \tilde{\varphi}\leq A|x|+B-\alpha\cdot x-\beta$ in $\C^n$. Let 
\begin{equation*}
\psi_{\epsilon}=\max\{\tilde{\varphi}, (A+\epsilon)|x|+B-\alpha\cdot x -\beta -C_{\epsilon}\},
\end{equation*}
where $C_{\epsilon}$ is chosen so that $\psi_{\epsilon}=\tilde{\varphi}$ for $|x|<R$
and $\psi_{\epsilon}= (A+\epsilon)|x|+B-\alpha\cdot x -\beta -C_{\epsilon}$ for $|x|\geq R'>R$. 
Then we have
\begin{eqnarray*}
\int_{|x|<R, y\in K} dd^{c}\varphi \wedge (dd^{c}|x|)^{n-1} &=&
\int_{|x|<R, y\in K} dd^{c}\tilde{\varphi} \wedge (dd^{c}|x|)^{n-1} \\
&\leq &
\int_{|x|<R', y\in K} dd^{c}\psi_{\epsilon} \wedge (dd^{c}|x|)^{n-1} \\
&=&
\int_{|x|=R', y\in K} d^{c}\psi_{\epsilon} \wedge (dd^{c}|x|)^{n-1} \\
&=&
(A+\epsilon)\int_{|x|<R', y\in K} dd^{c}|x|\wedge (dd^{c}|x|)^{n-1}.
\end{eqnarray*}
It is straightforward to verify that $d(d^{c}|x|\wedge (dd^{c}|x|)^{n-1})=0$ for $x\neq 0$
and so, by Stokes' theorem, 
\begin{equation*}
(A+\epsilon)\int_{|x|<R', y\in K} dd^{c}|x|\wedge (dd^{c}|x|)^{n-1}=
(A+\epsilon)\int_{|x|=1, y\in K} d^{c}|x|\wedge (dd^{c}|x|)^{n-1}.
\end{equation*}
The lemma thus follows with $A$ replaced by $A+\epsilon$; letting $\epsilon\to 0$ the lemma follows.
\end{proof}

\begin{proof}[Proof of Theorem~\ref{fund}]
(i): Notice first that $d^{c}|x|^2=2|x|d^{c}|x|$ and that 
$dd^{c}|x|^2=2|x|dd^{c}|x| + 2d|x|\wedge d^{c}|x|$ so that 
\begin{equation*}
d^{c}|x|^2\wedge (dd^{c}|x|^2)^{n-2} = 2^{n-1}|x|^{n-1} d^{c}|x| \wedge (dd^{c}|x|)^{n-2}.
\end{equation*}
Using this and the fact that $T$ is periodic in the imaginary directions, we get by Stokes' theorem
that
\begin{equation*}
\frac{1}{R^{n-1}}\int_{|x|<R, y\in [0,2\pi]^n} T\wedge (dd^{c}|x|^2)^{n-1} \quad \quad \quad \quad \quad \quad
\end{equation*}
\begin{eqnarray*}
\quad \quad \quad \quad &=& \frac{1}{R^{n-1}}\int_{|x|=R, y\in [0,2\pi]^n} T\wedge d^{c}|x|^2\wedge (dd^{c}|x|^2)^{n-2}\\
&=& 2^{n-1}\int_{|x|=R, y\in [0,2\pi]^n} T\wedge d^{c}|x|\wedge (dd^{c}|x|)^{n-2}\\
&=& 2^{n-1}\int_{|x|<R, y\in [0,2\pi]^n} T\wedge (dd^{c}|x|)^{n-1}.
\end{eqnarray*}
Since $T$ and $dd^{c}|x|$ are positive it follows that this last integral is increasing with $R$.
The function 
\begin{equation*}
\tilde{y}\mapsto \int_{|x|<R, y\in [0,2\pi]^n} T(x+i(y+\tilde{y}))\wedge (dd^{c}|x|)^{n-1}
\end{equation*}
is constant since $T$ is periodic and it follows that
\begin{equation*}
\int_{|x|<R, y\in [0,2\pi]^n} T(x+iy)\wedge (dd^{c}|x|)^{n-1}=
\int_{|x|<R, y\in [0,2\pi]^n} dd^c R(x)\wedge (dd^{c}|x|)^{n-1},
\end{equation*}
where $R$ is the Ronkin function associated to $f$. It is well-known that $R(x)$ is convex on $\R^n$ and satisfies $R(x)\leq A|x| + B$
for some constants $A$ and $B$. Thus, part (i) follows from Lemma~\ref{bob-lelong}.

\smallskip

(ii): As in the proof of (i) we see that
\begin{equation*}
\int_{|x|<R',\, y\in [0,2\pi]^n} T^r\wedge (dd^{c}|x|^2)^{n-1} =
(2R')^{n-1}\int_{|x|<R',\, y\in [0,2\pi]^n} dd^c R^r(x)\wedge (dd^{c}|x|)^{n-1},
\end{equation*}
where 
\begin{equation*}
R^r(x):=\frac{1}{(2\pi)^n}\int_{\tilde{y}\in [0,2\pi]^n} \log |f(x+r+i\tilde{y})| d\tilde{y}
\end{equation*}
is the translated Ronkin function.
This function is convex and has the same asymptotic behavior at infinity as $R(x)$, i.e., $R^r(x)\leq A|x| + B_r$,
where $A$ is independent of $r$.
Hence, (ii) follows from Lemma~\ref{bob-lelong}.
\end{proof}

\section{The current $\check{T}_f$}\label{sec:T}

Let $f$ be an exponential polynomial in $\C^n$ and let $\Delta_f$ be the associated Newton polytope.
Let, as before, $T_f:=dd^c\log |f|$ be the Lelong current associated with $f$. Recall that since $T_f$ is periodic in each 
imaginary direction with period $2\pi$ we can view $T_f$ as a current on $\C^n/(2\pi i\mathbb{Z})^n$. 
Thus, when considering the action of $T_f$,
or any other current with the same periodicity property, 
on a test form $\xi$ in $\C^n$ we may assume that the projection of $\textrm{supp}\, \xi$ on $i\R^n$ is contained in 
$i[0,2\pi]^n$.
For $r\in \R^n$, let as above $T^r_f$ be the 
translate $T^r_f(z)=T_f(z+r)$. Abusing notation somewhat, we define the current $\check{T}^R_f$ on $\C^n$ by
\begin{equation*}
\check{T}^R_f := \frac{1}{\kappa^{n-1}(R)}\int_{|r|<R}  T_f^r\, dr,
\end{equation*}
where $dr$ is the Lebesgue measure on $\R^n$. In terms of an $(n-1,n-1)$-test form $\xi$ on $\C^n$ this means that
\begin{equation*}
\int\check{T}^R_f \wedge \xi = \frac{1}{\kappa^{n-1}(R)}\int_{|r|<R} \int T_f^r \wedge \xi\, dr
= \frac{1}{\kappa^{n-1}(R)}\int_{|r|<R} \int_{z\in\C^n} T_f(z)\wedge \xi(z-r)\, dr.
\end{equation*}
Notice that $\check{T}_f^R$ is a positive $(1,1)$-current on $\C^n$ that has the same periodicity properties as $T_f$.
Let $\xi$ be a test form with support contained in $\mathbb{B}^n(0,s) + i[0,2\pi]^n$. Then
\begin{eqnarray*}
|\int\check{T}_f^R \wedge \xi| &\leq & \sup |\xi| 
\frac{1}{\kappa^{n-1}(R)}\int_{|r|<R} \int_{\substack{|\real\, z-r|<s \\ \imag\, z\in [0,2\pi]^n}} T_f(z)\wedge (dd^c|z|^2)^{n-1}\, dr \\
&\leq & \sup |\xi|\frac{\textrm{Vol}_n(\mathbb{B}^n(0,s))}{\kappa^{n-1}(R)}\int_{\substack{|\real\, z|<R+s \\ \imag\, z\in [0,2\pi]^n}}
T_f(z)\wedge (dd^c|z|^2)^{n-1} .
\end{eqnarray*}
It thus follows from Theorem~\ref{fund} (i) that the family of currents $\{\check{T}_f^R\}$ is bounded 
and hence has accumulation points.
In particular there is a sequence $\{R_j\}$ tending to $+\infty$ as $j\to \infty$, 
such that $\lim_{j\to \infty} \check{T}^{R_j}_f$ exists and defines a non-negative $(1,1)$-current. 
It is straightforward to check that the coefficients of $\lim_j \check{T}^{R_j}_f$ are independent of $\real\, z$.
We will now see that the \emph{trace mass} (defined below) of $\lim_j \check{T}^{R_j}_f$, as long as it exists, 
is independent of the sequence $\{R_j\}$.

If $\mu$ is a $(1,1)$-current in $\C^n$, periodic in the imaginary directions and
such that the coefficients of $\mu$ are independent of $\real\, z$, then we define the trace mass of $\mu$ as follows:
\begin{equation*}
Tr(\mu):=\int_{\substack{|\real\, z|<1 \\ \imag\, z\in [0,2\pi]^{n}}} \mu(z) \wedge (dd^c|z|^2)^{n-1}.
\end{equation*}
It follows from, e.g., \cite[Proposition~1.14, Ch.\ III, \S 1.B]{agbook} 
that if $\mu$ is non-negative and $Tr(\mu)=0$, then $\mu=0$ on
$\{z; |\real\, z|<1, \imag\, z\in [0,2\pi]^{n}\}$ and hence, $\mu=0$ everywhere.

\begin{theorem}\label{mass}
Let $f$ be an exponential polynomial in $\C^n$ with associated Newton polytope $\Delta_f$. If 
$\lim_{j\to \infty} \check{T}^{R_j}_f$ exists, then
\begin{equation*}
Tr(\lim_{j\to \infty} \check{T}^{R_j}_f) =
n! 4^{n-1} \textrm{Vol}_n(\B^n) V_{n-1}(\Delta_f,\B^n)/\textrm{Vol}_{n-1}(\B^{n-1}).
\end{equation*}
\end{theorem} 

\begin{proof}
Plugging into the definition we get that
\begin{eqnarray*}
Tr(\lim_{j\to \infty} \check{T}^{R_j}_f) &=&
\lim_{j\to \infty} \frac{1}{\kappa^{n-1}(R_j)} \int_{|r|<R} \int_{\substack{|\real\, z-r|<1 \\ \imag\, z\in[0,2\pi]^n}}
T_f(z)\wedge (dd^c |z|^2)^{n-1}\, dr \\
&=:& \lim_{j\to \infty} I(R_j). 
\end{eqnarray*}
A straightforward computation using Fubini's theorem shows that
\begin{equation*}
I(R_j)\leq \frac{\textrm{Vol}_n(\B^n)}{\kappa^{n-1}(R_j)}\int_{\substack{|\real\, z|<R_j+1 \\ \imag\, z\in[0,2\pi]^n}}
T_f(z)\wedge (dd^c|z|^2)^{n-1}.
\end{equation*}
Noticing that $dd^c |z|^2=2dd^c|x|^2$ and computing as in the first part of the proof of Theorem~\ref{fund} we obtain
\begin{eqnarray*}
\int_{\substack{|\real\, z|<R_j+1 \\ \imag\, z\in[0,2\pi]^n}}T_f(z)\wedge (dd^c|z|^2)^{n-1} 
&=& (4(R_j+1))^{n-1} \int_{\substack{|\real\, z|<R_j+1 \\ \imag\, z\in[0,2\pi]^n}}T_f(z)\wedge (dd^c|x|)^{n-1} \\
&=& (4(R_j+1))^{n-1} \int_{\substack{|x|<R_j+1 \\ y\in[0,2\pi]^n}} dd^c R_f(x)\wedge (dd^c|x|)^{n-1},
\end{eqnarray*}
where $R_f$ is the Ronkin function. 
We have thus showed that
\begin{equation*}
I(R_j)\leq \textrm{Vol}_n(\B^n)\frac{4^{n-1}(R_j+1)^{n-1}}{\kappa^{n-1}(R_j)}
\int_{\substack{|x|<R_j+1 \\ y\in [0,2\pi]^n}} dd^c R_f(x)\wedge (dd^c|x|)^{n-1}.
\end{equation*}
In a similar way we also get
\begin{equation*}
I(R_j)\geq \textrm{Vol}_n(\B^n)\frac{4^{n-1}(R_j-1)^{n-1}}{\kappa^{n-1}(R_j)}
\int_{\substack{|x|<R_j-1 \\ y\in [0,2\pi]^n}} dd^c R_f(x)\wedge (dd^c|x|)^{n-1}.
\end{equation*}
By Lemma~\ref{bob-lelong} and the discussion following it, it thus follows that $\lim_{j\to \infty} I(R_j)$ exists and that
\begin{equation}\label{klut}
\lim_{j\to \infty}I(R_j)=\frac{\textrm{Vol}_n(\B^n) 4^{n-1}}{\textrm{Vol}_{n-1}(\B^{n-1})}
\int_{\R^n + i[0,2\pi]^n} dd^c R_f(x)\wedge (dd^c|x|)^{n-1}
\end{equation}
independently of the sequence $\{R_j\}$. 
The theorem thus follows from \eqref{kul2}.
\end{proof}

\subsection{The case $\textrm{dim}\, \Delta_f=1$.}\label{ssec:1dim}

If $\textrm{dim}\, \Delta_f=1$, then as in the introduction we may write 
\begin{equation*}
f(z)=e^{\alpha_0\cdot z} \sum_{j=0}^{\ell} c_j \left(e^{\beta\cdot z}\right)^{k_j}
=e^{\alpha_0\cdot z} P(e^{\beta\cdot z})
\end{equation*}
for some $\alpha_0,\beta \in \mathbb{Z}^n$, some integers $0=k_0<k_1< \cdots<k_{\ell}$, and 
the one-variable polynomial $P(w)=\sum_j c_j w^{k_j}$. Let $\{a_j\}$ be the distinct zeros of $P$
and let $d_j$ be the multiplicity of $a_j$. Then
\begin{equation*}
T_f=\sum_j d_j [\beta\cdot z=\log a_j],
\end{equation*}
where we consider $\log$ as a multivalued function and $[\beta\cdot z=\log a_j]$ is the current of integration
over $\{\beta\cdot z=\log a_j\}$.

\begin{example}
Let $f(z)=1+e^{z_1+z_2}$. Then we may take $\alpha_0=(0,0)$, $\beta=(1,1)$, and $P(w)=1+w$, 
which has the simple zero $w=-1$. In view of the following proposition, which in particular proves \eqref{eq:1dimformel}, we have
\begin{equation*}
\check{T}_f=\frac{dx_1+dx_2}{\sqrt{2}}\wedge [y_1+y_2=\arg(-1)]. 
\end{equation*}
\hfill \qed
\end{example}

\begin{proposition}\label{edgeprop}
Let $f$ be as above and let $\xi$ be a test form with support in $\B^n(0,S)+i[0,2\pi]^n$. Then
\begin{multline*}
\lim_{R\to \infty} \frac{1}{\kappa^{n-1}(R)} \int_{|r|<R} \int_z T_f(z)\wedge \xi(z-r)\, dr \\
= \sum_j d_j \int_{\beta\cdot y=\arg a_j} \xi(x+iy)\wedge (\beta\cdot dx)/|\beta|.
\end{multline*}

Moreover, let $L^M=\{\lambda\beta;\, \lambda\in [-M/|\beta|, M/|\beta|]\}$, let $C$ be an $n-1$-dimensional
cone in the subspace $\{x\in\R^n;\, x\cdot \beta=0\}$, and let $D^M$ be the Minkowski sum of $L^M$ and $C$.
Then, for any fixed $M$ such that
$M>\max_j \{\big|\log |a_j| \big|/|\beta| + S\}$, we have that
\begin{multline*}
\lim_{R\to \infty} \frac{1}{\kappa^{n-1}(R)} \int_{\substack{|r|<R \\ r\in D^M}} \int_z T_f(z)\wedge \xi(z-r)\, dr = \\
\frac{\textrm{Vol}_{n-1}(C\cap\B^{n-1})}{\textrm{Vol}_{n-1}(\B^{n-1})}\sum_j d_j \int_{\beta\cdot y=\arg a_j} \xi(x+iy)\wedge (\beta\cdot dx)/|\beta|.
\end{multline*}
\end{proposition}

Notice that it is part of the statement that the limits exist. We choose the orientation
of $\{\beta\cdot y=\arg a_j\}$ so that the current $d(\beta \cdot x)\wedge[\beta\cdot y=\arg a_j]$ becomes 
positive.

\begin{proof}
The first part follows from the second; simply take $C=\{x\in\R^n;\, x\cdot \beta=0\}$.
To prove the second part,
recall that $T_f=\sum_j d_j [\beta\cdot z=\log a_j]$, where we consider $\log$ multivalued. 
It is thus sufficient to prove that 
\begin{multline*}
\lim_{R\to \infty} \frac{1}{\kappa^{n-1}(R)} \int_{\substack{|r|<R \\ r\in D^M}} \int_{\beta\cdot z=\log a_j} \xi(z-r)\, dr = \\
\frac{\textrm{Vol}_{n-1}(C\cap\B^{n-1})}{\textrm{Vol}_{n-1}(\B^{n-1})}\int_{\beta\cdot y=\arg a_j} \xi(x+iy)\wedge (\beta\cdot dx)/|\beta|,
\end{multline*}
where we consider $\arg$ multivalued.
To compute the integral on the left hand side we make the following change of coordinates: Let $\omega_j$,
$j=1,\ldots,n-1$ be orthonormal vectors in the subspace $\{x\in\R^n;\, x\cdot\beta=0\}$ such that
$(\omega_1,\ldots,\omega_{n-1},\beta/|\beta|)$ is a positively oriented orthonormal basis for $\R^n$.
Let $B$ be the matrix with $\omega_1,\ldots,\omega_{n-1}$ as its first $n-1$ rows
and with $\beta/|\beta|$ as its last. We then put $\zeta=Bz$ and $\rho=Br$. We get that
$D^M=\{\rho\in\R^n;\, |\rho_n|<M, \, \rho'\in C'\}$, where $\rho'=(\rho_1,\ldots,\rho_{n-1})$ and 
$C'$ is a cone in $\R^{n-1}$ such that $\textrm{Vol}_{n-1}(C'\cap\B^{n-1})=\textrm{Vol}_{n-1}(C\cap\B^{n-1})$.
We thus get
\begin{multline}\label{korv}
\lim_{R\to \infty} \frac{1}{\kappa^{n-1}(R)} \int_{\substack{|r|<R \\ r\in D^M}} \int_{\beta\cdot z=\log a_j} \xi(z-r)\, dr  \\
= \lim_{R\to \infty} \frac{1}{\kappa^{n-1}(R)} \int_{\substack{|\rho|<R \\ |\rho_n|<M \\ \rho'\in C'}} 
\int_{\zeta_n |\beta|=\log a_j} \xi(\zeta-\rho)\, d\rho.  
\end{multline}
Letting $\zeta'=(\zeta_1,\ldots,\zeta_{n-1})$, we can rewrite the inner integral on the right hand side as
\begin{equation*}
\int_{\zeta'\in\C^n} \xi(\zeta'-\rho',|\beta|^{-1}\log a_j - \rho_n) = 
\int_{\zeta'\in\C^n} \xi(\zeta',|\beta|^{-1}\log a_j - \rho_n)
\end{equation*}
showing that it is independent of $\rho'$. A straightforward computation using Fubini's theorem then
shows that the right hand side of \eqref{korv} equals
\begin{equation*}
\lim_{R\to \infty} \frac{\textrm{Vol}_{n-1}(C'\cap\B^{n-1})}{\kappa^{n-1}(R)} \int_{|\rho_n|<M}  (R^2-\rho_n^2)^{(n-1)/2}
\int_{\zeta'\in\C^n} \xi(\zeta',|\beta|^{-1}\log a_j - \rho_n)\, d\rho_n,
\end{equation*}
which by, e.g., dominated convergence equals
\begin{equation}\label{korv2}
\frac{\textrm{Vol}_{n-1}(C'\cap\B^{n-1})}{\textrm{Vol}_{n-1}(\B^{n-1})} \int_{|\rho_n|<M}
\int_{\zeta'\in\C^n} \xi(\zeta',|\beta|^{-1}\log a_j - \rho_n)\, d\rho_n.
\end{equation}
Since the inner integral in \eqref{korv2} is $0$ if $||\beta|^{-1}\log |a_j| - \rho_n|>S$
it follows from the choice of $M$ that the requirement $|\rho_n|<M$ in \eqref{korv2} 
is superfluous; we may just as well integrate over all of $\R$. Letting $t=|\beta|^{-1}\log |a_j| - \rho_n$
we see that \eqref{korv2} equals
\begin{multline*}
\frac{\textrm{Vol}_{n-1}(C'\cap\B^{n-1})}{\textrm{Vol}_{n-1}(\B^{n-1})} \int_{t\in\R}
\int_{\zeta'\in\C^n} \xi(\zeta',t+i\arg a_j/|\beta|)\, d t \\
=\frac{\textrm{Vol}_{n-1}(C'\cap\B^{n-1})}{\textrm{Vol}_{n-1}(\B^{n-1})} \int_{\imag\, \zeta_n=\arg a_j/|\beta|} \xi(\zeta)\wedge d(\real\, \zeta_n).
\end{multline*}
Changing back to the original $z$-coordinates and recalling that $\textrm{Vol}_{n-1}(C'\cap\B^{n-1})=\textrm{Vol}_{n-1}(C\cap\B^{n-1})$
the result follows.
\end{proof}

Proposition~\ref{edgeprop} allows us to make the following definition for exponential polynomials with
one-dimensional Newton polytope.
\begin{equation*}
\check{T}_f=\lim_{R\to \infty} \frac{1}{\kappa^{n-1}(R)}\int_{|r|<R} T_f^r\, dr.
\end{equation*}
Moreover, with $f$ given as above we get the explicit formula \eqref{eq:1dimformel}.

\subsection{The case $\textrm{dim}\, \Delta_f > 1$.}\label{ssec:kil} 
The purpose of this section
is to prove Theorem~\ref{Thuttsats}' below. Together with Proposition~\ref{edgeprop} it shows that 
$\check{T}_f:=\lim_{R\to\infty} (1/\kappa^{n-1}(R)) \int_{|r|<R} T_f^r\, dr$ exists
and is given by a quite explicit formula in terms of $\Delta_f$. 
Recall that $f(z)=\sum_{\alpha\in A}c_{\alpha}e^{\alpha\cdot z}$ and
that $\Delta_f=\textrm{conv}\, A$. For each edge $\Gamma$ of $\Delta_f$ we let, as in the introduction, 
\begin{equation*}
f_{\Gamma}(z)=\sum_{\alpha\in A\cap \Gamma} c_{\alpha}e^{\alpha\cdot z}.
\end{equation*}

\begin{example}\label{ex}
Let $f(z)=1+e^z_1+e^z_2$, cf.\ Figure~\ref{figur1}. Then $\Delta_f$ is the convex hull of $\{(0,0), (1,0), (0,1)\}$,
which has the three edges $\Gamma_1=\textrm{conv}\{(0,0), (1,0)\}$, $\Gamma_2=\textrm{conv}\{(0,0), (0,1)\}$, 
and $\Gamma_3=\textrm{conv}\{(1,0), (0,1)\}$. The external angle of $\Gamma_j$ is $1/2$;
indeed, if $\Gamma$ is any edge of any $2$-dimensional polytope in $\R^2$, then the external angle 
of $\Gamma$ is $1/2$. We have that
\begin{equation*}
f_{\Gamma_1}(z)=1+e^{z_1}, \quad f_{\Gamma_2}(z)=1+e^{z_2}, \quad f_{\Gamma_3}(z)=e^{z_1} + e^{z_2},
\end{equation*}
and by Theorem~\ref{Thuttsats} and \eqref{eq:1dimformel},
\begin{eqnarray*}
\check{T}_f &=& \frac{1}{2} \check{T}_{f_{\Gamma_1}} + \frac{1}{2} \check{T}_{f_{\Gamma_2}}
+\frac{1}{2} \check{T}_{f_{\Gamma_3}} \\
&=&
\frac{1}{2}\Big(dx_1\wedge [y_1=\arg(-1)] + dx_2\wedge [y_2=\arg(-1)] + \\
& & + \frac{dx_2-dx_1}{\sqrt{2}}\wedge [y_2-y_1=\arg(-1)]\Big). 
\end{eqnarray*}
\hfill \qed
\end{example}

We now fix an edge $\Gamma$ of $\Delta_f$. As in the introduction and the previous section, we write 
$\Gamma\cap A=\{\alpha_0,\ldots,\alpha_0+k_{\ell}\beta\}$,
$f_{\Gamma}(z)=e^{\alpha_0\cdot z} P(e^{\beta\cdot z})$, and we let $\{a_j\}$ be the distinct zeros of $P$ with multiplicities
$\{d_j\}$. Let $L_M=\{\lambda\beta;\, -M/|\beta|\leq\lambda\leq M /|\beta|\}$, let $N$ be the normal cone of $\Gamma$,
let $N^{\epsilon}$ be the corresponding smaller cone, cf.\ \eqref{Nepsilon}, and 
let $D_M^{\epsilon}=N^{\epsilon}+L_M$ be the Minkowski sum.

\begin{proposition}\label{edgeprop2}
Let $f$ and $f_{\Gamma}$ be as above and let $\xi$ be a test form in $\C^n$. Then, for $\epsilon>0$ sufficiently small and 
any $M>0$ we have
\begin{multline*}
\lim_{R\to\infty} \frac{1}{\kappa^{n-1}(R)}\int_{\substack{|r|<R \\ r\in D_M^{\epsilon}}}
\int_z T_f(z)\wedge \xi(z-r)\, dr \\
= \lim_{R\to\infty} \frac{1}{\kappa^{n-1}(R)}\int_{\substack{|r|<R \\ r\in D_M^{\epsilon}}}
\int_z T_{f_{\Gamma}}(z)\wedge \xi(z-r)\, dr.
\end{multline*}
\end{proposition} 

In the following lemma we will, for every $x\in D_M^{\epsilon}$, write $x=x'+x''$ for (unique)
$x'\in N^{\epsilon}$ and $x''\in L_M$.

\begin{lemma}\label{uppsklemma}
There are constants $K>0$ and $d>0$ independent of $\epsilon >0$ such that the following holds:
If $z\in \C^n$ is such that $\real\, z =x=x'+x''\in D_M^{\epsilon}$ and 
\begin{equation}\label{stor}
\big|\beta\cdot z - \log a_j\big|\geq K^{1/d} e^{-\epsilon |x'|/d}, \quad \forall j,
\end{equation}
where we consider $\log$ as a multivalued function, then $|f(z)-f_{\Gamma}(z)|<|f_{\Gamma}(z)|$.
\end{lemma} 

\begin{remark}
Notice that \eqref{stor} precisely means that the distance from $z$ to $f_{\Gamma}^{-1}(0)$ 
is greater than or equal to $K^{1/d} e^{-\epsilon |x'|/d}/|\beta|$.
\end{remark}

\begin{proof}
Since $f$ and $f_{\Gamma}$ are periodic in the imaginary directions we may assume that $\imag\, z\in [0,2\pi]^n$.

The image of $D_M^{\epsilon}+i[0,2\pi]^n$ under the map $z\mapsto \beta\cdot z$ is a set $Q\subset \C$ of the form
\begin{equation*}
Q:=M|\beta| [-1, 1] + i 2\pi [k_1,k_2]
\end{equation*}
for some $k_j\in \mathbb{Z}$. 
Hence, for $z\in D_M^{\epsilon}+i[0,2\pi]^n$ we have that
\begin{multline}\label{rot}
f_{\Gamma}(z)=e^{\alpha_0\cdot z} P(e^{\beta\cdot z}) = e^{\alpha_0\cdot z} C\prod_j (e^{\beta\cdot z}- a_j)^{d_j} \\
= e^{\alpha_0\cdot z} C\prod (\beta\cdot z-\log a_j)^{d_j}\cdot \psi(\beta\cdot z),
\end{multline}
where $C\in\C$ is a constant, $\psi$ is a holomorphic non-vanishing function in a neighborhood of $Q$, 
and the last product is over all $j$ and all values of $\log a_j$ such that $\log a_j\in Q$. Let $d$ be the sum of the
associated $d_j$'s and let $m=\inf_Q |\psi|>0$.
From \eqref{rot} and \eqref{stor} it thus follows that
\begin{equation*}
|f_{\Gamma}(z)| \geq e^{\alpha_0\cdot x} m|C| \prod K^{d_j/d} e^{-\epsilon |x'|d_j/d}=
e^{\alpha_0\cdot x} m|C| K e^{-\epsilon |x'|}.
\end{equation*}

Since $x'\in N^{\epsilon}$, we have 
$x'\cdot (\alpha_0-\alpha)\geq \epsilon |x'| |\alpha_0-\alpha|\geq \epsilon |x'|$ since 
$|\alpha_0-\alpha|\geq 1$ for any $\alpha\in A\setminus \Gamma$. Hence,

\begin{eqnarray*}
\left|\frac{f(z)-f_{\Gamma}(z)}{e^{\alpha_0\cdot z}}\right| &\leq &
\sum_{\alpha\in A\setminus \Gamma} |c_{\alpha}| e^{(\alpha-\alpha_0)\cdot x} \leq
e^{-\epsilon |x'|} \sum_{\alpha\in A\setminus \Gamma} |c_{\alpha}| e^{(\alpha-\alpha_0)\cdot x''} \\
&\leq & 
e^{-\epsilon |x'|} \sum_{\alpha\in A\setminus \Gamma} |c_{\alpha}| e^{M|\alpha-\alpha_0|}.
\end{eqnarray*}
The lemma therefore follows if we take
\begin{equation*}
K>\frac{1}{m|C|} \sum_{\alpha\in A\setminus \Gamma} |c_{\alpha}| e^{M|\alpha-\alpha_0|}.
\end{equation*}
\end{proof}

\begin{lemma}\label{deltalemma}
Let $D=D'\times D''\subset \C^{n-1}\times\C$ be a domain and let $h, \tilde{h}\in \hol(\overline{D})$. Assume that
$\tilde{h}(z)/z_n^k$ is a non-vanishing holomorphic function
and that $|h(z)-\tilde{h}(z)|<|\tilde{h}(z)|$ for all $z\in D$ such that $|z_n|\geq \delta$.
Then, if $\delta>0$ is sufficiently small and $\xi$ is a test form in $D$, 
\begin{equation*}
\left|\int T_h\wedge \xi - \int T_{\tilde{h}} \wedge \xi \right| \leq
\delta C_{\xi} \int_{h^{-1}(0)\cap \pi^{-1}(\pi (\textrm{supp}\, \xi))} (dd^c |z|^2)^{n-1},
\end{equation*}
where $C_{\xi}$ is a constant depending on the $C^1$-norm of $\xi$ and $\pi\colon D'\times D''\to D'$ is the 
natural projection.
\end{lemma}

\begin{proof}
Notice first that the zero set $Z$ of $h$ is within a distance $<\delta$ from $\{\zeta_n=0\}=\tilde{h}^{-1}(0)$.
In particular, $\pi\restriction_{Z}$ is proper. Moreover, from Rouche's theorem it follows that for each fixed
$z'\in D'$, the function $z_n\mapsto h(z',z_n)$ has precisely $k$ zeros with multiplicity. Thus,
$(\pi\restriction_Z)_* T_h = k[\zeta_n=0]=T_{\tilde{h}}$, which means that if $\tilde{\xi}$ is a test form in $D'$
and $\xi_0:=\pi^*\tilde{\xi}$ then
\begin{equation}\label{strut}
\int T_h\wedge\xi_0=\int T_{\tilde{h}}\wedge \xi_0.
\end{equation}

Let now $\xi$ be an arbitrary test form in $D$. By Taylor expanding in the $z_n$-direction we may write
\begin{equation}\label{strut2}
\xi=\xi_0 + \xi_1\wedge dz_n + \xi_2\wedge d\bar{z}_n + z_n\xi_3 + \bar{z}_n \xi_4,
\end{equation}
where the $\xi_j$'s are smooth and have the following properties: $\xi_0$ is independent of $z_n$, contains no
$dz_n$ or $d\bar{z}_n$, and has support in $\pi^{-1}(\pi (\textrm{supp}\, \xi))$; $\xi_1$ and $\xi_2$ are 
test forms in $D$ with the same support as $\xi$; $\xi_3$ and $\xi_4$ contain no $dz_n$ or $d\bar{z}_n$
and have supports in $\pi^{-1}(\pi (\textrm{supp}\, \xi))$. Moreover, the $C^0$-norms of $\xi_0$, $\xi_3$, and $\xi_4$ 
can be estimated by the 
$C^1$-norm of $\xi$, whereas the $C^0$-norm of $\xi_1$ and $\xi_2$ can be estimated by the $C^0$-norm of $\xi$. 

Notice that \eqref{strut} holds for $\xi_0$ as in \eqref{strut2}. Now,
\begin{equation*}
\left|\int T_h\wedge z_n\xi_3\right|\leq k\int_{h=0} |z_n\xi_3| (dd^c |z|^2)^{n-1} \leq
\delta C_{\xi_3} \int_{h^{-1}(0)\cap \pi^{-1}(\pi (\textrm{supp}\, \xi))} (dd^c |z|^2)^{n-1},
\end{equation*}
and a similar estimate holds for $\bar{z}_n\xi_4$. To see that $\xi_1\wedge dz_n$ and $\xi_2\wedge d\bar{z}_n$
also satisfy a similar estimate we use that $T_h$ is $d$-closed so that
\begin{equation*}
\int T_h\wedge \xi_1\wedge dz_n = \pm \int T_h \wedge z_n d\xi_1
\end{equation*}
and similarly for $\xi_2\wedge d\bar{z}_n$. The lemma follows.
\end{proof}

\begin{proof}[Proof of Proposition~\ref{edgeprop2}]
By periodicity of $f$ and $f_{\Gamma}$ we may assume that $\xi$ has support in $\B(0,S)+i[0,2\pi]^n$ for some
$S>0$. Notice that it is sufficient to prove that
\begin{multline}\label{penna}
\lim_{R\to\infty} \frac{1}{\kappa^{n-1}(R)}\int_{\substack{R'<|r|<R \\ r\in D_M^{\epsilon}}}
\int_z T_f(z)\wedge \xi(z-r)\, dr \\
= \lim_{R\to\infty} \frac{1}{\kappa^{n-1}(R)}\int_{\substack{R'<|r|<R \\ r\in D_M^{\epsilon}}}
\int_z T_{f_{\Gamma}}(z)\wedge \xi(z-r)\, dr
\end{multline}
for any fixed $R'>0$. We choose $R'$ as follows. 
Given $0<\epsilon'<\epsilon$ we let $R'>0$ be such that if $|r|>R'$ and $r\in D_M^{\epsilon}$,
then $\B^n(r,S)\subset D_{M+S}^{\epsilon'}$; clearly such $R'$ exists.

Let $B$ be the matrix in the proof of Proposition~\ref{edgeprop} and let $\zeta=Bz$ and $\rho=Br$
be the changes of coordinates of the same proof.
In the $\rho$-coordinates we have that
\begin{equation*}
D_{M+S}^{\epsilon'}=\{\rho\in \R^n;\, |\rho_n|\leq M+S, \, \rho'=(\rho_1,\ldots,\rho_{n-1})\in C^{\epsilon'}\},
\end{equation*}
where $C^{\epsilon'}$ is a cone in $\R^{n-1}$.

Let $\xi_{\rho}(\zeta)=\xi(\zeta-\rho)$; then $\xi_{\rho}$ has support contained in $\B^n(\rho,S)+i E$ for some 
bounded set $E = E'\times E''\subset \R^{n-1}\times \R$. Letting $I\subset \R$ be the
interval $[-M-S, M+S]$ and $\zeta'=(\zeta_1,\ldots,\zeta_{n-1})$ it is straightforward to verify that,
for $\rho\in D_M^{\epsilon}$ and $|\rho|>R'$,
\begin{equation}\label{suppxi}
\textrm{supp}\, \xi_{\rho}\subset \left((C^{\epsilon'}+iE')\times (I+iE'')\right) \cap
\{\zeta;\, \big| |\real\, \zeta'| - |\rho|\big| < M+2S \}.
\end{equation} 
Denote the set on the right hand side by $D_{\rho}=D'_{\rho}\times D''\subset \C^{n-1}\times \C$
and notice that $D_{\rho}$ is a uniformly bounded set for all $\rho$.

Now, let $\zeta\in D_{\rho}$, where $\rho\in D_M^{\epsilon}$ and $|\rho|>R'$. Then $\real\, \zeta\in D_{M+S}^{\epsilon'}$
by \eqref{suppxi}. Hence, by Lemma~\ref{uppsklemma}, there are constants $K>0$ and $d>0$ such that if
\begin{equation*}
\textrm{dist}(\zeta,f_{\Gamma}^{-1}(0))\geq K^{1/d} e^{-\epsilon' |\real\,\zeta'|/d},
\end{equation*}
then $|f(\zeta)-f_{\Gamma}(\zeta)|<|f_{\Gamma}(\zeta)|$.
Moreover, since $\zeta\in D_{\rho}$ it follows from \eqref{suppxi} that 
$-|\rho|+M+2S>-|\real\,\zeta'|$. Hence, if
\begin{equation*}
\textrm{dist}(\zeta,f_{\Gamma}^{-1}(0))\geq K^{1/d} e^{-\epsilon'|\rho|/d} e^{\epsilon'(M+2S)/d}/|\beta|,
\end{equation*}
then $|f(\zeta)-f_{\Gamma}(\zeta)|<|f_{\Gamma}(\zeta)|$. Therefore there is, by Lemma~\ref{deltalemma},
a constant $C_{\xi}$ (depending on the $C^1$-norm of $\xi$) such that
\begin{eqnarray*}
\left|\int (T_f(z)-T_{f_{\Gamma}}(z))\wedge \xi(z-r)\right| &=&
\left|\int (T_f(\zeta)-T_{f_{\Gamma}}(\zeta))\wedge \xi_{\rho}(\zeta)\right| \\
&\leq &
e^{-\epsilon'|\rho|/d} e^{\epsilon'(M+2S)/d} C_{\xi}\int_{f^{-1}(0)\cap D_{\rho}} (dd^c|\zeta|^2)^{n-1}.
\end{eqnarray*} 
for $r\in D_M^{\epsilon}$ and $|r|>R'$, i.e., $\rho\in D_M^{\epsilon}$ and $|\rho|>R'$. 
Notice that, by Theorem~\ref{fund} (ii), the integral $\int_{f^{-1}(0)\cap D_{\rho}} (dd^c|\zeta|^2)^{n-1}$
is bounded by a constant independent of $\rho$. Hence, since $|\rho|=|r|$,
\begin{multline*}
\frac{1}{\kappa^{n-1}(R)}\int_{\substack{R'<|r|<R \\ r\in D_M^{\epsilon}}}
\left|\int (T_f(z)-T_{f_{\Gamma}}(z))\wedge \xi(z-r)\right| \, dr\\
\lesssim
e^{\epsilon'(M+2S)/d} \frac{1}{\kappa^{n-1}(R)}\int_{\substack{R'<|r|<R \\ r\in D_M^{\epsilon}}} e^{-\epsilon'|r|/d}\, dr.
\end{multline*}
The right hand side goes to $0$ as $R\to \infty$ and so \eqref{penna} follows since the limit on the right hand side of
\eqref{penna} exists by Proposition~\ref{edgeprop}.
\end{proof}

\medskip

\noindent {\bf Theorem~\ref{Thuttsats}'}
Let $f$ be an exponential polynomial with associated Newton polytope $\Delta_f$. Then
\begin{equation}\label{penna4}
\lim_{R\to\infty} \check{T}_f^{R} = \sum_{\Gamma \in\mathcal{F}_1(\Delta_f)}
\gamma_{\Delta_f}(\Gamma) \check{T}_{f_{\Gamma}}.
\end{equation}

\medskip

\begin{proof}
Let $\{R_j\}$ be a sequence tending to $\infty$ such that $\lim_{j\to \infty} \check{T}_f^{R_j}$ exists.
Let $\xi$ be a test form in $\C^n$; we may assume that its support is contained in $\B(0,S)+i[0,2\pi]^n$.
Let $D_{M,\Gamma}^{\epsilon}=L_{M,\Gamma} + N_{\Gamma}^{\epsilon}$ be as above and set
$\mathcal{D}_M^{\epsilon}:=\bigcup_{\Gamma} D_{M,\Gamma}^{\epsilon}$. Then
\begin{eqnarray}\label{penna3}
\int \check{T}_f^{R_j}\wedge\xi &=& 
\frac{1}{\kappa^{n-1}(R_j)} \int_{\substack{|r|<R_j \\ r\in \mathcal{D}_M^{\epsilon}}} \int_z T_f(z)\wedge \xi(z-r)\, dr \\
& & +\frac{1}{\kappa^{n-1}(R_j)} \int_{\substack{|r|<R_j \\ r\notin \mathcal{D}_M^{\epsilon}}} 
\int_z T_f(z)\wedge \xi(z-r)\, dr. \nonumber
\end{eqnarray}

Recall that if $\Gamma\neq\Gamma'$, then 
$N_{\Gamma}^{\epsilon}\cap N_{\Gamma'}^{\epsilon}=\{0\}$ for any $\epsilon>0$. It follows that 
$D_{M,\Gamma}^{\epsilon}\cap D_{M,\Gamma'}^{\epsilon}$ is a bounded set and thus
\begin{equation*}
\lim_{R\to\infty}
\frac{1}{\kappa^{n-1}(R)} \int_{\substack{|r|<R \\ r\in D_{M,\Gamma}^{\epsilon}\cap D_{M,\Gamma'}^{\epsilon}}} 
\int_z T_f(z)\wedge \xi(z-r)\, dr = 0.
\end{equation*}
Hence, the first integral on the right hand side of \eqref{penna3} equals
\begin{equation*}
\sum_{\Gamma \in \mathcal{F}_1(\Delta_f)} 
\frac{1}{\kappa^{n-1}(R_j)} \int_{\substack{|r|<R_j \\ r\in D_{M,\Gamma}^{\epsilon}}} \int_z T_f(z)\wedge \xi(z-r)\, dr
+ \varrho(R_j),
\end{equation*}
where $\varrho$ is a function such that $\lim_{R\to\infty}\varrho(R)=0$.
By Propositions~\ref{edgeprop} and \ref{edgeprop2} we have that 
\begin{equation*}
\lim_{R\to\infty}
\frac{1}{\kappa^{n-1}(R)} \int_{\substack{|r|<R \\ r\in D_{M,\Gamma}^{\epsilon}}} \int_z T_f(z)\wedge \xi(z-r)\, dr =
\frac{\textrm{Vol}_{n-1}(N_{\Gamma}^{\epsilon}\cap\B^{n-1})}{\textrm{Vol}_{n-1}(\B^{n-1})} \int \check{T}_{f_{\Gamma}} \wedge \xi
\end{equation*}
if $\epsilon>0$ is sufficiently small and $M$ is sufficiently large.

Now, from \eqref{penna3} it follows that, for $\epsilon>0$ sufficiently small, we have
\begin{equation*}
\lim_{R_j\to\infty} \check{T}_f^{R_j} =
\sum_{\Gamma \in\mathcal{F}_1(\Delta_f)} 
\frac{\textrm{Vol}_{n-1}(N_{\Gamma}^{\epsilon}\cap\B^{n-1})}{\textrm{Vol}_{n-1}(\B^{n-1})} \check{T}_{f_{\Gamma}} +
\mu_{\epsilon},
\end{equation*}
where $\mu_{\epsilon}$ is a non-negative $(1,1)$-current.
The term in the sum on the right hand side corresponding to $\Gamma$ has 
the limit $\gamma_{\Delta_f}(\Gamma)\check{T}_{f_{\Gamma}}$ as $\epsilon\to 0$; cf.\ \eqref{extang}. It follows that 
$\mu=\lim_{\epsilon\to 0}\mu_{\epsilon}$ exists and thus also is a non-negative $(1,1)$-current.

We claim that $\mu=0$. Given the claim the theorem follows since then $\lim_{R_j\to\infty} \check{T}_f^{R_j}$
equals the right hand side of \eqref{penna4} for any choice of sequence $\{R_j\}$.
To prove the claim we will show that $Tr(\mu)=0$, i.e., that the trace mass of
$\lim_{j\to \infty} \check{T}^{R_j}_f$ and the trace mass of the right hand side of \eqref{penna4}
are equal. By Theorem~\ref{mass} we get that
\begin{equation*}
Tr(\lim_{j\to \infty} \check{T}^{R_j}_f) =
n! 4^{n-1} \textrm{Vol}_n(\B^n) V_{n-1}(\Delta_f,\B^n)/\textrm{Vol}_{n-1}(\B^{n-1})
\end{equation*}
and that
\begin{equation*}
Tr(\check{T}_{f_{\Gamma}}) =
n! 4^{n-1} \textrm{Vol}_n(\B^n) V_{n-1}(\Gamma,\B^n)/\textrm{Vol}_{n-1}(\B^{n-1}).
\end{equation*}
Since $\gamma_{\Gamma}(\Gamma)=1$ we get from \eqref{kul} that
$V_{n-1}(\Gamma,\B^n)=\textrm{Vol}_{n-1}(\B^{n-1})\textrm{Vol}_1(\Gamma)/n$. It then follows from \eqref{kul} that
\begin{eqnarray*}
Tr\left(\sum_{\Gamma \in\mathcal{F}_1(\Delta_f)} 
\gamma_{\Delta}(\Gamma) \check{T}_{f_{\Gamma}}\right) 
&=& 
n! 4^{n-1} \textrm{Vol}_n(\B^n) n^{-1}\sum_{\Gamma \in\mathcal{F}_1(\Delta_f)}
\gamma_{\Delta_f}(\Gamma) \textrm{Vol}_1(\Gamma)  \\
&=&
n! 4^{n-1} \textrm{Vol}_n(\B^n) V_{n-1}(\Delta_f,\B^n)/\textrm{Vol}_{n-1}(\B^{n-1}) \\
&=& 
Tr(\lim_{j\to \infty} \check{T}^{R_j}_f)
\end{eqnarray*}
and we are done.
\end{proof}

\section{The Ronkin type function}\label{sec:ronkin}

As before, for each edge $\Gamma$ of $\Delta_f$ we choose $\alpha_{\Gamma}, \beta_{\Gamma}\in \mathbb{Z}^n$
and $k^{\Gamma}_j\in \mathbb{N}$ such that $\Gamma\cap A=\{\alpha_{\Gamma}, \alpha_{\Gamma}+k^{\Gamma}_1\beta_{\Gamma},\cdots\}$.
Recall also that $f_{\Gamma}(z)/e^{\alpha_{\Gamma}\cdot z}$ is a one-variable polynomial in $e^{\beta_{\Gamma}\cdot z}$
with distinct zeros $\{a_{\Gamma,j}\}$ and multiplicities $\{d_{\Gamma,j}\}$, cf.\ the introduction and the
beginning of Section~\ref{ssec:kil}. We define the matrix $S_f$ as follows
\begin{equation}\label{Smatris}
S_f := \frac{1}{2} \sum_{\Gamma\in\mathcal{F}_1(\Delta_f)} \textrm{Vol}_1(\Gamma) \gamma_{\Delta_f}(\Gamma)
\frac{\beta_{\Gamma} \beta_{\Gamma}^t}{|\beta_{\Gamma}|^2}.
\end{equation}

\medskip

\noindent {\bf Example~\ref{ex}, continued.} In the situation of Example~\ref{ex} we have that 
$\textrm{Vol}_1(\Gamma_1)=\textrm{Vol}_1(\Gamma_2)=1$, $\textrm{Vol}_1(\Gamma_3)=\sqrt{2}$,
and we may take $\beta_{\Gamma_1}=(1,0)^t$, $\beta_{\Gamma_2}=(0,1)^t$, and $\beta_{\Gamma_3}=(-1,1)^t$.
A simple computation then shows that
\begin{equation*}
S_f=\frac{1}{4\sqrt{2}}
\left[\begin{array}{cc}
1+\sqrt{2} & -1 \\
-1         & 1+\sqrt{2}
\end{array}\right].
\end{equation*}
\hfill \qed

\medskip

\begin{theorem}\label{ronkin}
Let $\Phi\colon \R\to\R$ be the continuous piecewise linear function such that $\Phi(0)=0$
and such that the slope of $\Phi$ in the interval $2\pi[k-1,k]$ is $k$ for every $k\in\mathbb{Z}$.
Let $b_{\Gamma, j}=\arg a_{\Gamma, j}\cap [0,2\pi)$ and let
\begin{equation*}
\varphi_f(y):=2\pi  \sum_{\Gamma\in\mathcal{F}_1(\Delta_f)}
\frac{\gamma_{\Delta_f}(\Gamma)}{|\beta_{\Gamma}|} \sum_j d_{\Gamma, j} 
\Big(\Phi(\beta_{\Gamma}\cdot y - b_{\Gamma, j}) + \frac{b_{\Gamma, j}-\pi}{2\pi} \beta_{\Gamma}\cdot y \Big).
\end{equation*}
Then 
\begin{itemize}
\item[(i)] $dd^c \varphi_f = \check{T}_f$, 
\item[(ii)] $y\mapsto \varphi(y) - y^tS_fy$ is periodic, where $y$ is considered as a column vector.
\end{itemize}
\end{theorem}

The function $\varphi_f$ is our analogue of the Ronkin function.

\begin{proof}
By Theorem~\ref{Thuttsats}' and Proposition~\ref{edgeprop} we have that
\begin{equation*}
\check{T}_f= \sum_{\Gamma\in \mathcal{F}_1(\Delta_f)} \gamma_{\Delta_f}(\Gamma)
\sum_j d_{\Gamma, j} \frac{\beta_{\Gamma}\cdot dx}{|\beta|}\wedge [\beta_{\Gamma}\cdot y=\arg a_{\Gamma,j}],
\end{equation*}
where $\arg$ is multivalued. Using that $\Phi''=\sum_{k\in \mathbb{Z}} \delta_{2\pi k}$
in the sense of distributions, where
$\delta_{2\pi k}$ is the Dirac mass at $2\pi k$, it is straightforward to verify that $dd^c\varphi_f=\check{T}_f$.

It is also a straightforward computation to see that for all $t\in\R$ and all $k\in \mathbb{Z}$
\begin{equation*}
\Phi(t+2\pi k) - \Phi(t) = \pi k^2 + t k + \pi k.
\end{equation*}
Using this one verifies that 
\begin{equation*}
\varphi_f(y+\lambda) - \varphi_f(y) = (y+\lambda)^t S_f (y+\lambda) - y^t S_f y,
\end{equation*}
for $\lambda=2\pi \ell$, $\ell\in \mathbb{Z}^n$. Hence, $y\mapsto \varphi_f(y) - y^tS_fy$ is periodic.
\end{proof}

\noindent {\bf Example~\ref{ex}, continued.}
In the situation of Example~\ref{ex} we get that
\begin{equation*}
\varphi_f(y)=\pi \big(\Phi(y_1-\pi) + \Phi(y_2-\pi) + \frac{1}{\sqrt{2}}\Phi(y_2-y_1-\pi) \big).
\end{equation*}
\hfill \qed

\medskip

Since $y\mapsto \varphi_f(y)-y^tS_fy$ is periodic it follows that the gradient of $\varphi_f$,
$\nabla\varphi_f=(\partial\varphi_f/\partial y_1,\ldots,\partial\varphi_f/\partial y_n)^t$, satisfies
\begin{equation*}
\nabla\varphi_f(y+\lambda) - \nabla\varphi_f(y) = 2S_f (y+\lambda) - 2S_fy=2S_f\lambda,
\end{equation*}
where $\lambda=2\pi \ell$, $\ell\in \mathbb{Z}^n$. Letting $L_{S_f}$ be the lattice in $\R^n$ generated over $\mathbb{Z}$ by 
the columns of $4\pi S_f$ we see that we can view $\nabla\varphi_f$ as a function 
$\R^n/(2\pi \mathbb{Z})^n \to \R^n/L_{S_f}$.

\begin{proposition}
The function $\nabla\varphi_f\colon \R^n/(2\pi \mathbb{Z})^n \to \R^n/L_{S_f}$ is locally constant
outside $\textrm{supp}\, \check{T}_f=\R^n+i\mathcal{H}_f$ and induces an injective map from the set of connected components of
$(\R^n\setminus \mathcal{H}_f)/(2\pi \mathbb{Z})^n$
to $\R^n/L_{S_f}$.
\end{proposition}
\begin{proof}
In the complement of $\textrm{supp}\, \check{T}_f$ we have that $dd^c\varphi_f=0$ and hence, $\varphi_f$ is affine
there. Thus $\nabla\varphi_f$ is locally constant outside $\textrm{supp}\, \check{T}_f$.

Let us temporarily consider $\nabla\varphi_f$ as a function $\R^n\to\R^n$. Let $C_1$ and $C_2$ be two 
connected components of $\R^n\setminus \mathcal{H}_f$. Assume first that 
$\nabla\varphi_f\restriction_{C_1}=\nabla\varphi_f\restriction_{C_2}$.
After subtracting a linear term we may assume that $\nabla\varphi_f\restriction_{C_1}=\nabla\varphi_f\restriction_{C_2}=0$.
Thus, $\varphi_f\restriction_{C_1}=k_1$ and $\varphi_f\restriction_{C_2}=k_2$ are constant. Since $\varphi_f$ is convex $k_1$ and $k_2$
must be equal and so $C_1=C_2$. Assume now instead that 
$\nabla\varphi_f\restriction_{C_1}-\nabla\varphi_f\restriction_{C_2}\in L_{S_f}$,
i.e., that there is a $\lambda\in 2\pi\mathbb{Z}^n$ such that
\begin{equation*}
\nabla\varphi_f\restriction_{C_1}-\nabla\varphi_f\restriction_{C_2} = 2S_f\lambda.
\end{equation*} 
Let $\tilde{C}_1:=C_1-\lambda$; then $\tilde{C}_1$ is a connected component of $\R^n\setminus \mathcal{H}_f$
whose image in $\R^n/(2\pi\mathbb{Z}^n)$ is the same as the image of $C_1$ and
$\nabla\varphi_f\restriction_{C_1}-\nabla\varphi_f\restriction_{\tilde{C}_1}=2S\lambda$. Hence,
$\nabla\varphi_f\restriction_{\tilde{C}_1}=\nabla\varphi_f\restriction_{C_2}$ and so $\tilde{C}_1=C_2$. The images of $C_1$ and $C_2$ in
$\R^n/(2\pi\mathbb{Z}^n)$ are thus the same and the proposition follows.
\end{proof}

The function $\varphi_f$, now considered as a function on $\R^n_y$, is convex and piecewise affine and so its 
Legendre transform, $\hat{\varphi}_f$ has the same properties. In fact, $\hat{\varphi}_f$ gives rise to a 
convex subdivision of $\R^n$ that is dual to the subdivision given by $\varphi_f$; cf.\ \cite[Section~3]{PaRull}. 
In view of Theorem~\ref{ronkin} (ii) the subdivision given by $\varphi_f$ is in fact a subdivision of $\mathbb{T}^n$.
We claim that $\eta\mapsto \hat{\varphi}_f(\eta) - \eta^t(4S_f)^{-1}\eta$ is well-defined on $\R^n/L_{S_f}$;
thus the subdivision given by $\hat{\varphi}_f$ is a subdivision of the ``skew'' torus $\R^n/L_{S_f}$.
By the proposition above, the number of complement components of 
$(\R^n\setminus \mathcal{H}_f)/(2\pi \mathbb{Z})^n$ then equals the number of vertices of the 
subdivision of $\R^n/L_{S_f}$. To see the claim let $v\in L_{S_f}$ and write $v=4\pi S_f\ell$ for some $\ell\in \mathbb{Z}^n$.
Notice that $\varphi_f(y) - 4\pi y^tS_f\ell=\varphi(y-2\pi \ell) - 4\pi^2\ell^tS_f\ell$ for any $y\in\R^n$
by Theorem~\ref{ronkin} (ii). We get
\begin{eqnarray*}
\hat{\varphi}_f(\eta+v) &=& \hat{\varphi}_f(\eta+4\pi S_f\ell) = \sup_y \,\, (\eta+4\pi S_f\ell)\cdot y - \varphi_f(y) \\
&=& \sup_y \,\, \eta\cdot y -(\varphi_f(y) - 4\pi y^tS_f\ell) \\
&=& \sup_y\,\, \eta\cdot (y-2\pi\ell) - \varphi_f(y-2\pi\ell) +4\pi^2\ell^t S_f\ell + 2\pi\eta\cdot \ell \\
&=& \hat{\varphi}_f(\eta) + 4\pi^2\ell^t S_f\ell + 2\pi\eta\cdot \ell \\
&=& \hat{\varphi}_f(\eta) + v^t(4S_f)^{-1}v + 2\eta^t (4S_f)^{-1}v.
\end{eqnarray*} 
Hence, $\hat{\varphi}_f(\eta+v) - \hat{\varphi}_f(\eta)=(\eta+v)^t(4S_f)^{-1}(\eta+v) - \eta^t(4S_f)^{-1}\eta$
and the claim follows.

\bigskip


Let us now consider the Monge-Amp\`{e}re expression $(dd^c\varphi_f)^n$
and relate it to intersections of hyperplanes from $\mathcal{H}_f$. 
Let $\mathcal{I}_f$ be the set of points, repeated according to multiplicity, in $[0,2\pi)^n$ (or any other fundamental domain)
where $n$ hyperplanes from $\mathcal{H}_f$ meet transversely. To begin with, we consider
piecewise affine functions of the form $\psi_j(y):= \Phi(\beta_j\cdot y - b_j)$, where $\beta_j\in \mathbb{Z}^n$ and $b_j\in [0,2\pi)$;
we let $a_j\in \C$ be such that $\arg a_j\cap [0,2\pi)=b_j$.
By induction it is straightforward to show that
\begin{equation}\label{ok}
\bigwedge_{j=1}^k dd^c\psi_j(y) =\frac{1}{(2\pi)^k} \bigwedge_{j=1}^k (\beta_j\cdot dx) \wedge
[\beta_1\cdot y=\arg a_1,\ldots,\beta_k\cdot y=\arg a_k],
\end{equation}
where $[\beta_1\cdot y=\arg a_1,\ldots,\beta_k\cdot y=\arg a_k]$ is the current of integration over the 
affine set $\cap_1^k\{\beta_j\cdot y=\arg a_j\}$ oriented so that the right hand side becomes positive. Notice in particular
that $dd^c\psi_1\wedge \cdots \wedge dd^c\psi_k=0$ if the $\beta_j$'s are linearly dependent.

\begin{lemma}\label{stress}
Let $\beta_1,\ldots,\beta_n\in \mathbb{Z}^n$ be column vectors linearly independent over $\R$ 
and let $B$ be the matrix whose $j^{\textrm{th}}$ column
is $\beta_j$. 
Then the number of solutions to $B^ty=(\arg a_1,\ldots,\arg a_n)^t$ in $[0,2\pi)^n$ equals
$\det B$.
\end{lemma}
\begin{proof}
Consider the integral 
\begin{equation*}
I:=\int_{\B^n+i[0,2\pi)^n} dd^c \Phi(\beta_1\cdot y -b_1)\wedge\cdots\wedge dd^c\Phi(\beta_n\cdot y -b_n).
\end{equation*}
On one hand, by \eqref{ok}, $I=(2\pi)^{-n}\textrm{Vol}_n(\B^n) \det B$ times the number of 
solutions to $B^ty=(\arg a_1,\ldots,\arg a_n)^t$ in $[0,2\pi)^n$. 
On the other hand, 
\begin{eqnarray}\label{klot}
I &=& \int_{\B^n+i[0,2\pi)^n} \Big(dd^c \sum_j \Phi(\beta_j\cdot y-b_j)\Big)^n/n! \\
&=& \int_{\B^n+i[0,2\pi)^n} \Big(dd^c \sum_j \Phi(\beta_j\cdot y-b_j) +\frac{b_j-\pi}{2\pi}\beta_j\cdot y\Big)^n/n! \nonumber \\
&=& \int_{\B^n+i[0,2\pi)^n} \Big(dd^c \sum_j y^t \frac{\beta_j\beta_j^t}{4\pi}y + p_j(y)\Big)^n/n!, \nonumber
\end{eqnarray}
where $p_j(y)$ is periodic; the last equality follows as in the proof of Theorem~\ref{ronkin} (ii).
Notice that if $A$ is a form with constant coefficients then 
\begin{equation}\label{lut}
\int_{\B^n+i[0,2\pi)^n} A\wedge dd^cp_{j_1}(y)\wedge \cdots \wedge dd^cp_{j_k}(y) =0
\end{equation}
by Stokes' theorem. Hence we may remove $p_j(y)$ from the last integral in \eqref{klot} and
letting $S^j:=\beta_j\beta_j^t/4\pi$ we can thus write it as
\begin{eqnarray*}
\int_{\B^n+i[0,2\pi)^n} \Big( \sum_{j,k,\ell} S^j_{k\ell}dx_k\wedge dy_{\ell}/\pi\Big)^n/n! &=&
\frac{1}{\pi^n}\textrm{Vol}_n(\B^n)\det(\sum_j S^j) (2\pi)^n \\
&=&\textrm{Vol}_n(\B^n) \det(BB^t)/(2\pi)^n
\end{eqnarray*}
and the lemma follows.
\end{proof}

\begin{porism}
Let $\beta_1,\ldots,\beta_n$ and $B$ be as in the lemma above.
Then the number of integer points in the parallelogram spanned by 
the $\beta_j$ equals $\det B$.
\end{porism}

\begin{proof}
Let $T\colon \R_x^n+i\R_y^n\to \R_{\xi}^n+i\R_{\eta}^n$ be the map given by $T(x+iy)=Bx/2\pi + iBy/2\pi$
and let $Q=[0,1]^n\subset \R_x^n$. As in the proof of Lemma~\ref{stress} we see that
the number of integer points in the parallelogram spanned by the $\beta_j$'s equals
\begin{equation*}
\frac{(2\pi)^n}{\textrm{Vol}_n\,T(Q)}\int_{T(Q+i[0,2\pi)^n)} dd^c \Phi(2\pi\eta_1) \wedge \cdots\wedge dd^c\Phi(2\pi\eta_n)
\quad \quad \quad \quad
\end{equation*}  
\begin{equation*}
\quad \quad \quad \quad
=\frac{(2\pi)^n}{\textrm{Vol}_n\,T(Q)}\int_{Q+i[0,2\pi)^n} dd^c \Phi(B_1\cdot y) \wedge \cdots\wedge dd^c\Phi(B_n\cdot y),
\end{equation*}
where $B_j$ is the $j^{\textrm{th}}$ row of $B$. However, again as in the proof of Lemma~\ref{stress}, this last integral 
(without the coefficient) equals $(\det B)^2/(2\pi)^n$. Since $\textrm{Vol}_n\,T(Q)=\det B$ the result follows.
\end{proof}

As in the introduction, let $\Gamma_1,\ldots,\Gamma_N$ be an enumeration of the edges of $\Delta_f$,
set $a_{i,j}:=a_{\Gamma_i,j}$, $\beta_i:=\beta_{\Gamma_i}$, and let $B$ be the $n\times N$-matrix whose columns are the 
$\beta_{i}$'s.
For a subset $I=\{i_1,\ldots,i_n\}\subset \{1,\ldots,N\}$ we let $B_I$ be the matrix whose columns are
the $\beta_{i_k}$'s.

\begin{corollary}\label{genericCor}
Assume that $f$ is generic in the following sense: For each $j$, the function $f_{\Gamma_j}$
has no multiple zeros and at each point of $[0,2\pi)^n$ at most $n$ hyperplanes from $\mathcal{H}_f$ meet. Let $\mathcal{C}_f$
be the set of connected components of $(\R^n\setminus\mathcal{H}_f)/(2\pi\mathbb{Z})^n$. Then
\begin{equation}\label{jul}
|\mathcal{C}_f|=|\mathcal{I}_f|=\sum_{\substack{I\subset \{1,\ldots,N\} \\ |I|=n}} |\det B_I|
\end{equation} 
and this number is an upper bound for the number of complement components of $\mathcal{A}'_f$ considered on the torus.
\end{corollary}
\begin{proof}
The first equality of \eqref{jul} follows from \cite[Section~3]{toric}. The second equality is immediate from Lemma~\ref{stress}.
The final assertion follows from the paragraph preceding Figure~1 in the introduction.
\end{proof}


\begin{proof}[Proof of Theorem~\ref{estimate}]
Let $h(y):=\varphi_f(y) - y^tS_fy$
and let for simplicity $S:=S_f$. Then
\begin{equation*}
(dd^c\varphi_f)^n = \sum_{k=0}^n {n \choose k} \Big(\sum_{j,k}S_{jk}dx_j\wedge dy_k/\pi \Big)^k\wedge (dd^c h)^{n-k}.
\end{equation*}
Since $h$ is periodic by Theorem~\ref{ronkin} (ii) it follows as in \eqref{lut} that 
\begin{eqnarray*}
\frac{1}{\textrm{Vol}_n(\B^n)} \int_{\B^n+i[0,2\pi)^n} (dd^c\varphi_f)^n/n! &=& 
\frac{1}{\textrm{Vol}_n(\B^n)\pi^n n!} \int_{\B^n+i[0,2\pi)^n} \Big( \sum_{j,k}S_{jk}dx_j\wedge dy_k\Big)^n \\
&=& 
2^n\det S,
\end{eqnarray*}

On the other hand, a computation using the explicit form of $\varphi_f$ shows that
\begin{equation*}
(dd^c\varphi_f)^n/n!=(2\pi)^n \sum_{\substack{I\subset \{1,\ldots,N\} \\ |I|=n}}
\prod_{k=1}^n\frac{\gamma_{\Delta_f}(\Gamma_{i_k})}{|\beta_{i_k}|}
\sum_{j_1,\ldots,j_n}\bigwedge_{k=1}^nd_{i_k,j_k}dd^c\Phi(\beta_{i_k}\cdot y - b_{i_k,j}).
\end{equation*}
From Lemma~\ref{stress} and its proof it thus follows that 
\begin{equation*}
\frac{1}{\textrm{Vol}_n(\B^n)} \int_{\B^n+i[0,2\pi)^n} (dd^c\varphi_f)^n/n! =
\sum_{\substack{I\subset \{1,\ldots,N\} \\ |I|=n}}
\prod_{k=1}^n\frac{\gamma_{\Delta_f}(\Gamma_{i_k})}{|\beta_{i_k}|}
|\det B_I| \sigma_I,
\end{equation*}
where $\sigma_I$ is the number of intersection points of $\beta_{i_k}\cdot y =\arg a_{i_k,j_k}$, $k=1,\ldots,n$,
in $[0,2\pi)^n$ counted with multiplicity. Hence, 
\begin{equation*}
m \sum_{\substack{I\subset \{1,\ldots,N\} \\ |I|=n}} \sigma_I \leq
\frac{1}{\textrm{Vol}_n(\B^n)} \int_{\B^n+i[0,2\pi)^n} (dd^c\varphi_f)^n/n! \leq
M \sum_{\substack{I\subset \{1,\ldots,N\} \\ |I|=n}} \sigma_I
\end{equation*}
and the result follows.
\end{proof}

\noindent {\bf Example~\ref{ex}, continued.} In the situation of Example~\ref{ex} we have that 
$|\det B_I|=1$ for all $I\subset \{1,2,3\}$ with $|I|=2$ so the number of complement 
components of $\mathcal{H}_f$ as well as the number of its intersection points is $1+1+1=3$ by 
Corollary~\ref{genericCor} since $f$ is generic. Moreover, $\gamma_{\Delta_f}(\Gamma_i)=1/2$,
$|\beta_1|=|\beta_2|=1$, and $|\beta_3|=\sqrt{2}$ so, after a simple computation, the estimates of Theorem~\ref{estimate}
take the form
\begin{equation*}
1+\sqrt{2} \leq |\mathcal{I}_f| \leq 2+\sqrt{2}.
\end{equation*} 
\hfill \qed

\begin{example}
As a check that our computations are sound let us compute the trace mass of $\check{T}:=\check{T}_f$
(that we already know in view of Theorem~\ref{mass})
using the matrix $S:=S_f$. Let $h(y):=\varphi_f(y)-y^tSy$.
Then, by Theorem~\ref{ronkin} (i),
\begin{equation}\label{PL}
dd^ch = \check{T} - dd^c(y^tSy)=\check{T}-\frac{1}{\pi}\sum_{j,k}S_{jk}dx_j\wedge dy_k.
\end{equation}
Since $h$ is periodic by Theorem~\ref{ronkin} (ii) a straightforward computation shows that
\begin{eqnarray*}
Tr(\check{T}) &=& \int_{\substack{|\real\, z|<1 \\ \imag\, z\in [0,2\pi]^n}}
\check{T}\wedge (dd^c|z|^2)^{n-1} = 2^{n-1}\int_{\substack{|x|<1 \\ y\in [0,2\pi]^n}}
dd^c(y^tSy)\wedge (dd^c|x|^2)^{n-1} \\
&=&
2^{2n-1} (n-1)! \textrm{Vol}_n(\B^n) Tr(S),
\end{eqnarray*}
where $Tr(S)$ is the ordinary trace of the matrix $S$. Using that $Tr(\beta_{\Gamma}\beta_{\Gamma}^t/|\beta_{\Gamma}|^2)=1$
it follows from \eqref{Smatris} and \eqref{kul} that
\begin{equation*}
Tr(S)=\frac{n}{2\textrm{Vol}_{n-1}(\B^{n-1})}V_{n-1}(\Delta_f,\B^n).
\end{equation*}
Hence,
\begin{equation*}
Tr(\check{T})=n! 4^{n-1} \textrm{Vol}_n(\B^n) V_{n-1}(\Delta_f,\B^n)/\textrm{Vol}_{n-1}(\B^{n-1}),
\end{equation*}
which fits with Theorem~\ref{mass}.
\hfill \qed
\end{example}

\section{Cohomological interpretations}
Let $\check{T}$, $\varphi$, and $S=(S_{jk})_{jk}$ be the current, Ronkin type function, and matrix respectively, associated to an 
exponential polynomial $f$. Let also $h(y):=\varphi(y)-y^tSy$. 
We can write $T=\sum_{j,k}T_{jk}dx_j\wedge dy_k$ for certain distributions $T_{jk}$ that only depend on $y$;
set $T_j:=\sum_k T_{jk}dy_k$ so that $T=\sum_j dx_j\wedge T_j$. Let
\begin{equation*}
S_j:=\frac{1}{\pi}\sum_k S_{jk}dy_k.
\end{equation*}
Then we can interpret $S_j$ both as a global closed $1$-form on $\mathbb{T}^n$ and on $\C^n/(2\pi i\mathbb{Z})^n$. 
With the latter interpretation we have $dd^c(y^tSy)=\sum_jdx_j\wedge S_j$ and with the first we get, in view of \eqref{PL},
\begin{equation*}
\frac{1}{2\pi} d\frac{\partial h}{\partial y_j} = T_j - S_j, \quad j=1,\ldots,n.
\end{equation*}
Hence, $[T_j]=[S_j]$ as elements in $H^1(\mathbb{T}^n,\R)$. Let $\sigma_k$ be the path on $\mathbb{T}^n$
corresponding to $t\mapsto te_k$, $t\in [-\pi,\pi)$, where $\{e_j\}$ is the standard basis of $\R^n$, 
and let $[\sigma_k]\in H_1(\mathbb{T}^n)$ be the associated homology class.

\begin{proposition}
With the above notation we have that 
\begin{equation*}
S_{jk}=\frac{1}{2}[T_j]\cap [\sigma_k],
\end{equation*}
where $\cap$ is the cap product $H^1(\mathbb{T}^n,\R)\times H_1(\mathbb{T}^n)\to \R$.
\end{proposition}
\begin{proof}
This is merely a simple computation:
\begin{equation*}
[T_j]\cap [\sigma_k] = [S_j]\cap [\sigma_k] = \int_{\sigma_k} S_j = \frac{1}{\pi}\int_{-\pi}^{\pi}S_{jk}dt
=2S_{jk}.
\end{equation*}
\end{proof}

\smallskip

Let us finally notice that \eqref{PL} can be viewed as a Poincar\'{e}-Lelong type formula in $\C^n/(2\pi i\mathbb{Z})^n$
or $\mathbb{T}^n$
if we interpret $dd^c(y^tSy)$ as the Chern form of $\check{T}$. However, this ``Chern form'' is in general not
integer valued.

\end{document}